\documentclass{article}

\usepackage[left=1.2in,right=1.2in,top=1.5in,bottom=2in]{geometry}
\usepackage[english]{babel}
\usepackage{authblk}
\usepackage{amsfonts}
\usepackage{amsthm}
\usepackage{amsmath}
\usepackage{amssymb}
\usepackage{amsopn}
\usepackage{url}			
\usepackage{cite}		
\usepackage{subcaption}

\usepackage{booktabs}	
\usepackage{csvsimple}	
\usepackage{longtable}	
\usepackage{siunitx}	
\usepackage{etoolbox}	
\usepackage{multirow}   

\usepackage{algorithm}
\usepackage{algorithmic}

\usepackage{listings}
\usepackage[dvipsnames]{xcolor}

\lstdefinestyle{mystyle}{,   
    commentstyle=\color{OliveGreen},
    keywordstyle=\color{Blue},
    stringstyle=\color{BrickRed},
    basicstyle=\ttfamily\footnotesize,
    breakatwhitespace=false,         
    breaklines=true,        
    keepspaces=true,                 
    numbers=left,                    
    numbersep=10pt,                  
    showspaces=false,                
    showstringspaces=false,
    showtabs=false,                  
    tabsize=4
}

\usepackage{todonotes}

\newcommand{\R}{\mathbb{R}}

\newcommand{\norm}[1]{\left\lVert#1\right\rVert}
\newcommand{\normInf}[1]{\left\lVert#1\right\rVert_{\infty}}
\newcommand{\sqnorm}[1]{\left\lVert#1\right\rVert^{2}}

\newcommand{\drho}{\rho_{d}}
\newcommand{\prho}{\rho_{p}}
\newcommand{\grho}{\rho_{g}}

\newcommand{\HSD}{(HSD)}
\newcommand{\rHSD}{(rHSD)}

\newtheorem{theorem}{Theorem}

\newcommand{\revision}[2]{#2}

\title{\revision{blue}{Design and implementation of a modular\\ interior-point solver for linear optimization}}

\author[1]{Miguel F. Anjos}
\author[2]{Andrea Lodi}
\author[2]{Mathieu Tanneau\thanks{Mathieu Tanneau was supported by an excellence doctoral scholarship from FQRNT.}%
}
\affil[1]{School of Mathematics, University of Edinburgh}
\affil[2]{CERC Data Science, Polytechnique Montréal}

\date{}

\begin{document}

\maketitle

\begin{abstract}
 	This paper introduces the algorithmic design and implementation of Tulip, an open-source interior-point solver for linear optimization.
 	It implements \revision{red}{a regularized} homogeneous interior-point algorithm with multiple centrality corrections, and therefore handles unbounded and infeasible problems.
 	\revision{blue}{The solver is written in Julia, thus allowing for a flexible and efficient implementation:} Tulip's algorithmic framework is fully disentangled from linear algebra implementations \revision{blue}{and from a model's arithmetic}.
 	\revision{blue}{In particular}, this allows to seamlessly integrate specialized routines for structured problems.
 	Extensive computational results are reported.
 	We find that Tulip is competitive with open-source interior-point solvers on the \revision{blue}{H. Mittelmann's benchmark of barrier linear programming solvers}.
 	Furthermore, \revision{blue}{we design specialized linear algebra routines for structured master problems in the context of Dantzig-Wolfe decomposition.
 	These routines yield a tenfold speedup on large and dense instances that arise in power systems operation and two-stage stochastic programming, thereby outperforming state-of-the-art commercial interior point method solvers.
 	Finally, we illustrate Tulip's ability to use different levels of arithmetic precision by solving problems in extended precision.}
	
\end{abstract}

\section{Introduction}
\label{sec:intro}
	
	Linear programming (LP) algorithms have been around for over 70 years, and LP remains a fundamental paradigm in optimization.
	Indeed, although nowadays most real-life applications involve discrete decisions or non-linearities, the methods employed to solve them often rely on LP as their workhorse.
	Besides algorithms for mixed-integer linear programming (MILP), these include cutting-plane and outer-approximation algorithms that substitute a non-linear problem with a sequence of iteratively refined LPs \cite{Kelley1960, Westerlund1995, Mitchell2009}.
	Furthermore, LP is at the heart of classical decomposition methods such as Dantzig-Wolfe and Benders decompositions \cite{Dantzig1960, Benders1962}.
	Therefore, efficient and robust LP technology is instrumental to our ability to solve more involved optimization problems.

	Over the past few decades, interior-point methods (IPMs) have become a standard and efficient tool for solving LPs \cite{Wright1997, Gondzio2012_IPMreview}.
	While IPMs tend to overcome Dantzig's simplex algorithm on large-scale problems, the latter is well-suited for solving sequences of closely related LPs, by taking advantage of an advanced basis.
	Nevertheless, beyond sheer performance, it is now well recognized that a number of LP-based algorithms can further benefit from IPMs, despite their limited ability to warm start.
	In cutting plane algorithms, stronger cuts are often obtained by cutting off an interior point rather than an extreme vertex \cite{Bixby1992, Mitchell2000, Mitchell2009}.
	Similarly,
	\revision{blue}{IPMs have been successfully employed in the context of decomposition methods \cite{Elhedhli2004,Babonneau2009,NaoumSawaya2013,Munari2013,Gondzio2016_PDCGMlarge}, wherein}
	well-centered interior solutions typically provide a stabilization effect \cite{Gondzio1996_PDCGM, Rousseau2007,Gondzio2013_PDCGMnew}, thus reducing tailing-off and improving convergence.
	
\revision{red}{
	\subsection{Exploiting structure in IPMs}
}

	The remarkable performance of IPMs stems from both strong algorithmic foundations and efficient linear algebra.
	Indeed, the main computational effort of IPMs resides in the resolution, at each iteration, of a system of linear equations.
	Therefore, the efficiency of the underlying linear algebra has a direct impact of the method's overall performance.
	\revision{red}{Remarkably}, while most IPM solvers employ general-purpose sparse linear algebra routines, substantial speedups can be obtained by exploiting a problem's specific structure.
	\revision{red}{Nevertheless, successfully doing so requires (i) identifying a problem's structure and associated specialized linear algebra, (ii) integrating these custom routines within an IPM solver, and (iii) having a convenient and flexible  way for the user to convey structural information to the solver.
	The main contribution of our work is to simplify the latter two points.
	}
	
	\revision{red}{Numerous works have studied structure-exploiting IPMs, e.g., \cite{Birge1988,Hurd1988,Schultz1991,Choi1993,Jessup1994,Gondzio1997_GUB,Gondzio2003_IPMparLP,Gondzio2007_IPMparQP,Castro2017}.}
	For instance, block-angular matrices typically arise in stochastic programming when using scenario decomposition.
	In \cite{Birge1988} and later in \cite{Jessup1994}, the authors thus design specialized factorization techniques that outperform generic implementations.
	\revision{red}{Schultz et al. \cite{Schultz1991} design a specialized IPM for block-angular problems; therein, linking constraints are handled separately, thus allowing to decompose the rest of the problem.}
	Gondzio \cite{Gondzio1997_GUB} observed that the master problem in Dantzig-Wolfe decomposition possesses a block-angular structure.
	Similar approaches have been explored for network flow problems \cite{Choi1993}, \revision{red}{multi-commodity flow problems \cite{Gondzio2003_IPMparLP}, asset management problems \cite{Gondzio2007_IPMparQP}}, and for solving facility location problems \cite{Hurd1988,Castro2017}.
	
	The aforementioned works focus on devising specialized linear algebra for a particular structure or application.
	On the other hand, a handful of IPM codes that accommodate various linear algebra implementations have been developed.
	The OOQP software, developed by Gertz and Wright \cite{Gertz2003}, uses object-oriented design so that data structures and linear algebra routines can be tailored to specific applications.
	Motivated by large-scale stochastic programming, PIPS \cite{Lubin2011} incorporates a large share of OOQP's codebase, alongside specialized linear solvers for block-angular matrices.
	Nevertheless, to the best of the authors' knowledge, OOQP is no longer actively maintained, while current development on PIPS focuses on non-linear programming.\footnote{Personal communication with PIPS developers.}
	In a similar fashion, OOPS \cite{Gondzio2003_IPMparLP,Gondzio2007_IPMparQP,Gondzio2009_OOPS} implements custom linear algebra that can exploit arbitrary block matrix structures.
	We also note that both PIPS and OOPS are primarily intended for massive parallelism on high-performance computing infrastructure.
    \revision{red}{Furthermore, the BlockIP software \cite{Castro2016} is designed for block-angular convex optimization problems, and solves linear systems with a combination of Cholesky factorization and preconditioned conjugate gradient.
    Both OOPS and BlockIP can be accessed through SML \cite{Grothey2009} --which requires AMPL,}
    and are distributed under a closed-source proprietary license.
    
    \revision{blue}{
    Finally, while nowadays most optimization solvers are written in C or C++, users are increasingly turning to higher-level programming languages such as Python, Matlab or Julia, alongside a variety of modeling tools, e.g, Pyomo \cite{hart2011pyomo}, CVXPY \cite{diamond2016cvxpy}, YALMIP \cite{Lofberg2004}, JuMP \cite{Dunning2017_JuMP}, to mention a few.
    Thus, users of high-level languages often have to switch to a low-level language in order to implement performance-critical tasks such as linear algebra.
    This situation, commonly referred to as the ``two-language problem", hinders code development, maintenance, and usability.
    }\\

\subsection{Contributions and outline}
\label{sec:intro:subsec:outline}

    In this paper, we describe \revision{blue}{the design and implementation of a modular interior-point solver, Tulip.
    The solver is written in Julia \cite{Bezanson2017}, which offers several advantages.
    First, Julia combines both high-level syntax and fast performance, thus addressing the two-language problem.
    In particular, it offers built-in support for linear algebra, with direct access to dense and sparse linear algebra libraries such as BLAS, LAPACK and SuiteSparse \cite{SuiteSparse}.
    Second, the Julia ecosystem for optimization comprises a broad range of tools, from solvers' wrappers to modeling languages, alongside a growing and dynamic community of users.
    Finally, Julia's multiple dispatch feature renders Tulip's design fully flexible,}
    thus allowing to disentangle the IPM algorithmic framework from linear algebra implementations, \revision{blue}{and to solve problems in arbitrary precision arithmetic.}

	The remainder of the paper is structured as follows.	
	In Section \ref{sec:notations}, we introduce some notations and relevant definitions.
	
	In Section \ref{sec:ipm}, we describe the homogeneous self-dual embedding, and Tulip's \revision{red}{regularized} homogeneous interior-point algorithm.
	This feature contrasts with most IPM LP codes, namely, those that implement the almost-ubiquitous infeasible primal-dual interior-point algorithm \cite{Mehrotra1992}.
	The main advantage of the homogeneous algorithm is its ability to return certificates of primal or dual infeasibility.
	It is therefore better suited for use within cutting-plane algorithms or decomposition methods, wherein one may encounter infeasible or unbounded LPs.
	
	\revision{red}{In Section \ref{sec:KKT}, we highlight the resolution of linear systems within Tulip, which builds on black-box linear solvers.
	This modular design leverages Julia's multiple dispatch, thereby facilitating the integration of custom linear algebra with no performance loss due to using external routines.}
	
	\revision{blue}{The presolve procedure is described in Section \ref{sec:presolve} and,}
	in Section \ref{sec:implementation}, we provide further implementation details of Tulip, such as the treatment of variable bounds, default values of parameters, and \revision{red}{default linear solvers}.
	Tulip is publicly available \cite{Tulip} under an open-source license.
	It can be used as a stand-alone package in Julia, and through the solver-independent \revision{blue}{interface \texttt{MathOptInterface} \cite{Legat2020_MOI}}.
	
	In Section \ref{sec:res}, we report on \revision{blue}{three} sets of computational experiments.
	First, we compare Tulip to several open-source and commercial IPM solvers on \revision{blue}{a benchmark set of unstructured LP instances}.
	We observe that, using generic sparse linear algebra, Tulip is competitive with open-source IPM solvers.
	Second, we \revision{blue}{demonstrate Tulip's flexible design.
	We consider block-angular problems with dense linking constraints from two column-generation applications, for which we design specialized linear algebra routines.
    This implementation yields a tenfold speedup, thereby outperforming commercial solvers on large-scale instances.}
	\revision{blue}{Third, we show how extended precision can alleviate numerical difficulties, thus illustrating Tulip's ability to work in arbitrary precision arithmetic.}
	
	Finally, Section \ref{sec:conclusion} concludes the paper and highlights future research directions.

\section{Notations}
\label{sec:notations}

	We consider LPs in primal-dual standard form
	\begin{align}
		\label{eq:standard_LP}
		\begin{array}{rrl}
			(P) \ \ \ \displaystyle \min_{x} \ \ \ 
			& c^{T} x\\
			s.t. \ \ \ 
			& A x &= b,\\
			& x & \geq 0,
		\end{array}
		&
		\hspace*{1cm}
		\begin{array}{rrl}
			(D) \ \ \ \displaystyle \max_{y, s} \ \ \
			& b^{T} y \ \ \ \ \ \  \\
			s.t. \ \ 
			& A^{T}y + s &= c,\\
			& s & \geq 0,
		\end{array}
	\end{align}
	where $c, x, s \in \mathbb{R}^{n}$, $b, y \in \mathbb{R}^{m}$, and $A \in \mathbb{R}^{m \times n}$ is assumed to have full row rank.
	We follow the usual notations from interior-point literature, and write $X$ (resp. $S$) the diagonal matrix whose diagonal is given by $x$ (resp. $s$), i.e., $X := Diag(x)$ and $S := Diag(s)$.

	We denote $I$ the identity matrix and $e$ the vector with all coordinates equal to one; their respective dimensions are always obvious from context.
	The norm of a vector is written $\norm{\cdot}$ and, unless specified otherwise, it denotes the $\ell_{2}$ norm.
	
	A primal solution $x$ is feasible if $Ax = b$ and $x \geq 0$.
	A strictly feasible (or interior) solution is a primal feasible solution with $x > 0$.
	Similarly, a dual solution $(y, s)$ is feasible if $A^{T}y + s = c$ and $s \geq 0$, and strictly feasible if, additionally, $s > 0$.
	Finally, a primal-dual solution $(x, y, s)$ is optimal for \eqref{eq:standard_LP} if $x$ is primal-feasible, $(y,s)$ is dual-feasible, and their objective values are equal, i.e., $c^{T}x = b^{T}y$.
	
	A solution $(x, y, s)$ with $x, s \geq 0$ is strictly complementary if
	\begin{align}
		\label{eq:strcict_complementarity}
		\forall i \in \{1, \dots, n\}, \big( x_{i} s_{i} = 0 \text{ and } x_{i} + s_{i} > 0 \big).
	\end{align}
	The \revision{blue}{complementarity gap} is defined as $x^{T}s$.
	When $(x, y, s)$ is primal-dual feasible, the \revision{blue}{complementarity gap} equals the classical optimality gap, i.e., we have $x^{T}s = c^{T}x - b^{T}y$.
	
	For ease of reading, we assume, without loss of generality, that all primal variables are required to be non-negative.
	The handling of free variables and of variables with finite upper bound will be detailed in Section \ref{sec:implementation}.

\section{\revision{red}{Regularized homogeneous interior-point algorithm}}
\label{sec:ipm}

In this section, we describe the homogeneous self-dual formulation and algorithm.
Our implementation largely follows the algorithmic framework of \cite{Xu1996} and \cite{Andersen2000}, \revision{red}{combined with the primal-dual regularization scheme of \cite{Friedlander2012}}.
Consequently, we focus on the algorithm's main components, and refer to \revision{red}{\cite{Xu1996, Andersen2000,Friedlander2012}} for convergence proofs and theoretical results.
Specific implementation details will be further discussed in Section \ref{sec:implementation}.

\subsection{Homogeneous self-dual embedding}
\label{sec:ipm:hsd_model}

	The simplified homogeneous self-dual form was introduced in \cite{Xu1996}.
	It consists in reformulating the primal-dual pair \eqref{eq:standard_LP} as a single, self-dual linear program, which writes
	\begin{align}
        \label{eq:HSD:obj}
        (HSD) \ \ \ 
        \min_{x, y, \tau} \ \ \
        & 0 \\
        s.t. \ \ \ 
        \label{eq:HSD:dual}
        & -A^{T}y + c \tau \geq 0,\\
        \label{eq:HSD:primal}
        & Ax  - b \tau = 0,\\
        \label{eq:HSD:gap}
        & - c^{T}x + b^{T}y  \geq 0,\\
        \label{eq:HSD:domain}
        & x, \tau \geq 0,
    \end{align}
	where $\tau$ \revision{red}{is a scalar variable.
	Let $s$, $\kappa$ be the non-negative slacks associated to \eqref{eq:HSD:dual} and \eqref{eq:HSD:gap}, respectively.}
	A solution $(x, y, s, \tau, \kappa)$ is strictly complementary if
	\begin{align*}
		x_{i}s_{i} = 0, x_{i} + s_{i} > 0, \text{ and } \tau \kappa = 0, \tau + \kappa > 0.
	\end{align*}
	Problem $(HSD)$ is always feasible, has empty interior and, under mild assumptions, possesses a strictly complementary feasible solution \cite{Xu1996}.
	
	Let $(x^{*}, y^{*}, s^{*}, \tau^{*}, \kappa^{*})$ be a strictly complementary feasible solution for $(HSD)$.
	If $\tau^{*} >0$, then $(\frac{x^{*}}{\tau^{*}}, \frac{y^{*}}{\tau^{*}}, \frac{s^{*}}{\tau^{*}})$ is an optimal solution for the original problem \eqref{eq:standard_LP}.
	Otherwise, we have $\kappa^{*} >0$ and thus $c^{T}x^{*} - b^{T}y^{*} < 0$.
	In that case, the original problem (P) is infeasible or unbounded.
	If $c^{T}x^{*} < 0$, then (P) is unbounded and $x^{*}$ is an unbounded ray.
	If $-b^{T}y^{*} < 0$, then (P) is infeasible and $y^{*}$ is an unbounded dual ray.
	The latter is also referred to as a Farkas proof of infeasibility.
	Finally, if both $c^{T}x^{*} < 0$ and $-b^{T}y^{*} < 0$, then both (P) and (D) are infeasible.
	
\revision{red}{
\subsection{Regularized formulation}

Friedlander and Orban \cite{Friedlander2012} introduce an exact primal-dual regularization scheme for convex quadratic programs, which we extend to the HSD form.
The benefits of regularizations will be further detailed in Section \ref{sec:KKT}.
Importantly, rather than viewing $\HSD$ as a generic LP to which the regularization procedure of \cite{Friedlander2012} is applied, we exploit the fact that $\HSD$ is a self-dual embedding of $(P){-}(D)$, and formulate the regularization in the original primal-dual space.

Thus, we consider a \emph{single}, regularized, self-dual problem
\begin{align}
    \label{eq:rHSD:obj}
    (rHSD) \ \ \ 
    \min_{x, y, \tau} \ \ \
    & \prho (x - \bar{x})^{T}x + \drho (y - \bar{y})^{T}y + \grho (\tau - \bar{\tau}) \tau \\
    s.t. \ \ \ 
    \label{eq:rHSD:dual}
    & -A^{T}y + c \tau +  \prho (x - \bar{x}) \geq 0,\\
    \label{eq:rHSD:primal}
    & Ax  - b \tau + \drho (y - \bar{y}) = 0,\\
    \label{eq:rHSD:gap}
    & - c^{T}x + b^{T}y + \grho (\tau - \bar{\tau}) \geq 0,\\
    \label{eq:rHSD:domain}
    & x, \tau, \geq 0,
\end{align}
where $\rho_{p}, \rho_{d}, \rho_{g}$ are positive scalars, and $\bar{x} \in \R^{n}, \bar{y} \in \R^{m}, \bar{\tau} \in \R$ are given estimates of an optimal solution of $\HSD$.
We denote by $s, \kappa$ the non-negative slack variables of constraints \eqref{eq:rHSD:dual} and \eqref{eq:rHSD:gap}, respectively.
The first-order Karush-Kuhn-Tucker (KKT) conditions for $\rHSD$ can then be expressed in the following form:
\begin{align}
    \label{eq:KKT:rHSD:dual}
    \prho x - A^{T}y - s + c \tau &= \prho \bar{x},\\
    \label{eq:KKT:rHSD:primal}
    A x + \drho y - b \tau &= \drho \bar{y},\\
    \label{eq:KKT:rHSD:gap}
    -c^{T}x + b^{T}y + \grho \tau - \kappa &= \grho \bar{\tau},\\
    \label{eq:KKT:rHSD:xs}
    x_{j} s_{j} &= 0, & j=1, ..., n\\
    \label{eq:KKT:rHSD:tk}
    \tau \kappa &= 0,\\
    \label{eq:KKT:rHSD:domain}
    x, s, \tau, \kappa &\geq 0.
\end{align}

The correspondence between $\rHSD$ and \cite{Friedlander2012} follows from the fact that, up to a constant term, the objective function \eqref{eq:rHSD:obj} equals
\begin{align*}
    \frac{1}{2}
    \left(
        \prho \sqnorm{x - \bar{x}} + \drho \sqnorm{y - \bar{y}} + \grho \sqnorm{\tau - \bar{\tau}}
        + \prho \sqnorm{x} + \drho \sqnorm{y} + \grho \sqnorm{\tau}
    \right).
\end{align*}
Note that, for $\prho = \drho = \grho = 0$, the regularized problem $\rHSD$ reduces to $\HSD$.
Furthermore, Theorem \ref{thm:exact_reg} shows that, for positive $\prho, \drho, \grho$, the regularization is exact.

\begin{theorem}
\label{thm:exact_reg}
    Assume $\prho, \drho, \grho > 0$.
    Let $(x^{*}, y^{*}, \tau^{*})$ be a complementary optimal solution of $\HSD$, and let $(\bar{x}, \bar{y}, \bar{\tau}) = (x^{*}, y^{*}, \tau^{*})$ in the definition of $\rHSD$.
    Then, $(x^{*}, y^{*}, \tau^{*})$ is the unique optimal solution of $\rHSD$.
\end{theorem}
\begin{proof}The uniqueness of the optimum is a direct consequence of $\rHSD$ being a convex problem with strictly convex objective.

Next, we show that any feasible solution of $\rHSD$ has non-negative objective.
Let $(x, y, s, \tau, \kappa)$ be a feasible solution of $\rHSD$.
Substituting Eq. \eqref{eq:rHSD:dual}-\eqref{eq:rHSD:gap} into the objective \eqref{eq:rHSD:obj}, one obtains
\begin{align*}
    Z &= \prho (x - \bar{x})^{T}x + \drho (y - \bar{y})^{T}y + \grho (\tau - \bar{\tau}) \tau\\
    & = (A^{T}y + s - c\tau)^{T}x  + (b \tau - Ax)^{T}y + (c^{T}x - b^{T}y + \kappa) \tau\\
    & = x^{T}s + \tau \kappa \geq 0.
\end{align*}
Then, $(x^{*}, y^{*}, \tau^{*})$ is trivially feasible for $\rHSD$, and its objective value is $(x^{*})^{T}s^{*} + \tau^{*} \kappa^{*} = 0$.
Thus, it is optimal for $\rHSD$, which concludes the proof.
\qed
\end{proof}
}
	
\subsection{\revision{red}{Regularized homogeneous algorithm}}
\label{sec:ipm:hsd_algo}

	We now describe the \revision{red}{regularized} homogeneous interior-point algorithm.
	\revision{red}{%
	    Similar to \cite{Friedlander2012}, we apply a single Newton iteration to a sequence of problems of the form $\rHSD$ where, at each iteration, $\bar{x}, \bar{y}, \bar{\tau}$ are chosen to be the current primal-dual iterate.
	}
	
	Let $(x, y, s, \tau, \kappa)$ denote the current primal-dual iterate, with $(x, s, \tau, \kappa) >0$, and define the residuals
	\begin{align}
		r_{p}  &= b \tau - Ax,\\
		r_{d}  &= c \tau - A^{T}y - s,\\
		r_{g}  &= c^{T}x - b^{T}y + \kappa,
	\end{align}
	and the barrier parameter
	\[
		\mu = \dfrac{x^{T}s + \tau \kappa}{n+1}.
	\]
	
	\revision{red}{For given $\bar{x}, \bar{y}, \bar{\tau}$,}
	a search direction $(\delta_{x}, \delta_{y}, \delta_{s}, \delta_{\tau}, \delta_{\kappa})$ is computed by solving a Newton system of the form
	\revision{red}{
	\begin{align}
        \label{eq:Newton:dual}
		-\prho \delta_{x} + A^{T} \delta_{y} + \delta_{s} - c \delta_{\tau}
		    &= \eta \left( c \tau - A^{T}y - s + \prho (\bar{x} - x) \right),\\
		\label{eq:Newton:primal}
		A \delta_{x} + \drho \delta_{y} - b \delta_{\tau}
		    &= \eta \left( b \tau - Ax - \drho(y - \bar{y}) \right),\\
		\label{eq:Newton:gap}
	    -c^{T} \delta_{x} + b^{T} \delta_{y} + \grho \delta_{\tau} - \delta_{\kappa}
	        &= \eta \left( c^{T}x - b^{T}y + \kappa - \grho (\tau - \bar{\tau}) \right),\\
	    \label{eq:Newton:xs}
	    S \delta_{x} + X \delta_{s}
	        &= -XSe + \gamma \mu e,\\
	    \label{eq:Newton:tk}
		\kappa \delta_{\tau} + \tau \delta_{\kappa}
		    &= -\tau \kappa + \gamma \mu,
	\end{align}}
	where $\gamma$ and $\eta$ are non-negative scalars \revision{red}{whose values will be specified in Section \ref{sec:ipm:algo:direction}}.
	\revision{red}{%
	We evaluate the Newton system at $(\bar{x}, \bar{y}, \bar{\tau}) = (x, y, \tau)$, which yields
	\begin{align}
	    \label{eq:Newton:reg}
	    \begin{bmatrix}
	        -\prho I && A^{T} && I && -c && 0\\
	        A && \drho I && 0 && -b && 0\\
	        -c^{T} && b^{T} && 0 && \grho && -1\\
	        S && 0 && X && 0 && 0\\
	        0 && 0 && 0 && \kappa && \tau
	    \end{bmatrix}
	    \begin{bmatrix}
	        \delta_{x}\\
	        \delta_{y}\\
	        \delta_{s}\\
	        \delta_{\tau}\\
	        \delta_{\kappa}
	    \end{bmatrix}
	    =
	    \begin{bmatrix}
	        \eta r_{d}\\
	        \eta r_{p}\\
	        \eta r_{g}\\
	        -XSe + \gamma \mu e\\
	        -\tau \kappa + \gamma \mu
	    \end{bmatrix}
	    .
	\end{align}
	System \eqref{eq:Newton:reg} is identical to the Newton system obtained when solving $\HSD$ (see, e.g., \cite{Andersen2000}), except for the regularization terms that appear in the left-hand side.
	In particular, the right-hand side remains unchanged.
	}

\subsubsection{Starting point}

	We choose the following default starting point
	\[
		(x^{0}, y^{0}, s^{0}, \tau^{0}, \kappa^{0}) = (e, 0, e, 1, 1).
	\]
	This initial point was proposed in \cite{Xu1996}.
	Besides its simplicity, it has well-balanced \revision{blue}{complementarity} products, which are all equal to one.

\subsubsection{Search direction}
\label{sec:ipm:algo:direction}

	At each iteration, a search direction is computed using Mehrotra's predictor-corrector technique \cite{Mehrotra1992}, combined with Gondzio's multiple centrality corrections \cite{Gondzio1996_correction}.
	Following \cite{Andersen2000}, we adapt the original formulas of \cite{Mehrotra1992, Gondzio1996_correction} to account for the homogeneous embedding.

	First, the affine-scaling direction $(\delta^{\text{aff}}_{x}, \delta^{\text{aff}}_{y}, \delta^{\text{aff}}_{s}, \delta^{\text{aff}}_{\tau}, \delta^{\text{aff}}_{\kappa})$ is obtained 	by solving the Newton system
	\begin{align}
		\label{eq:NS_aff_dual}
		\revision{red}{-\prho \delta^{\text{aff}}_{x}} + A^{T} \delta^{\text{aff}}_{y} + \delta^{\text{aff}}_{s} - c \delta^{\text{aff}}_{\tau} &= r_{d},\\
		\label{eq:NS_aff_primal}
		 A \delta^{\text{aff}}_{x} \revision{red}{+ \drho \delta^{\text{aff}}_{y}} -b \delta^{\text{aff}}_{\tau} &= r_{p},\\
		\label{eq:NS_aff_opt}
	    -c^{T} \delta^{\text{aff}}_{x} + b^{T} \delta^{\text{aff}}_{y} \revision{red}{+ \grho \delta^{\text{aff}}_{\tau}} -\delta^{\text{aff}}_{\kappa} &= r_{g},\\
	    \label{eq:NS_aff_comp_xs}
	    S \delta^{\text{aff}}_{x} + X \delta^{\text{aff}}_{s} &= -XSe,\\
	    \label{eq:NS_aff_comp_tk}
		\kappa \delta^{\text{aff}}_{\tau} + \tau \delta^{\text{aff}}_{\kappa} &= -\tau \kappa,
	\end{align}
	which corresponds to \revision{red}{\eqref{eq:Newton:reg} with} $\eta=1$ and $\gamma=0$.
	Taking a full step ($\alpha=1$) would thus reduce both infeasibility and \revision{blue}{complementarity} gap to zero.
	However, doing so is generally not possible, due to the non-negativity requirement on $(x, s, \tau, \kappa)$.

	Consequently, a corrected search direction is computed, as proposed in \cite{Mehrotra1992}.
	The corrected direction hopefully enables one to make longer steps, thus reducing the total number of IPM iterations.
	Let $\eta = 1 - \gamma$, where
	\begin{align}
		\gamma = (1 - \alpha^{\text{aff}})^{2} \min\left( \revision{blue}{\gamma_{min}}, (1 - \alpha^{\text{aff}}) \right)
	\end{align}
	for some $\revision{blue}{\gamma_{min}} > 0$, and
	\begin{align}
		\alpha^{\text{aff}} = \max \left\{ 0 \leq \alpha \leq 1 \mid (x, s, \tau, \kappa) + \alpha (\delta^{\text{aff}}_{x}, \delta^{\text{aff}}_{s}, \delta^{\text{aff}}_{\tau}, \delta^{\text{aff}}_{\kappa}) \geq 0 \right\} .
	\end{align}
	The corrected search direction is then given by
	\begin{align}
		\label{eq:NS_cor_dual}
		\revision{red}{-\prho \delta_{x}} + A^{T} \delta_{y} + \delta_{s} - c \delta_{\tau}			&= \eta r_{d},\\
		\label{eq:NS_cor_primal}
		 A \delta_{x} \revision{red}{+ \drho \delta_{y}} -b \delta_{\tau} 							&= \eta r_{p},\\
		\label{eq:NS_cor_opt}
	    -c^{T} \delta_{x} + b^{T} \delta_{y} \revision{red}{+ \grho \delta_{\tau}} -\delta_{\kappa}   	&= \eta r_{g},\\
	    \label{eq:NS_cor_comp_xs}
	    S \delta_{x} + X \delta_{s}                   			&= -XSe + \gamma \mu e - \Delta_{x}^{\text{aff}} \Delta_{s}^{\text{aff}} e,\\
	    \label{eq:NS_cor_comp_tk}
		\kappa \delta_{\tau} + \tau \delta_{\kappa}   			&= -\tau \kappa + \gamma \mu - \delta_{\tau}^{\text{aff}} \delta_{\kappa}^{\text{aff}},
	\end{align}
	where $\Delta^{\text{aff}}_{x} = Diag(\delta^{\text{aff}}_{x})$ and $\Delta^{\text{aff}}_{s}=Diag(\delta^{\text{aff}}_{s})$.

\subsubsection{\revision{red}{Additional centrality corrections}}
	\revision{red}{Additional centrality corrections}
	aim at improving the centrality of the new iterate, i.e., to keep the complementary products well balanced.
	Doing so generally allows to make longer steps, thus reducing the total number of IPM iterations.
	We implement Gondzio's original technique \cite{Gondzio1996_correction}, with some modifications introduced in \cite{Andersen2000}.

	Let $\delta = (\delta_{x}, \delta_{y}, \delta_{s}, \delta_{\tau}, \delta_{\kappa})$ be the current search direction, $\alpha^{max}$ the corresponding maximum step size, and define
	\begin{align}
		(\bar{x}, \bar{y}, \bar{s}, \bar{\tau}, \bar{\kappa}) := (x, y, s, \tau, \kappa) + \bar{\alpha} (\delta_{x}, \delta_{y}, \delta_{s}, \delta_{\tau}, \delta_{\kappa}),
	\end{align}
	where $\bar{\alpha} := \min(1, 2 \alpha^{max})$ is a tentative step size.
	
	First, a soft target in the space of \revision{blue}{complementarity} products is computed as
	\begin{align}
		t_{j} &= 
		\left\{
		\begin{array}{cl}
			\mu_{l} - \bar{x}_{j} \bar{s}_{j} & \text{ if } \bar{x}_{j} \bar{s}_{j} < \mu_{l}\\
			0 & \text{ if } \bar{x}_{j} \bar{s}_{j} \in [\mu_{l}, \mu_{u}]\\
			\mu_{u} - \bar{x}_{j} \bar{s}_{j} & \text{ if } \bar{x}_{j} \bar{s}_{j} > \mu_{u}
		\end{array}
		\right.
		, \ \ 
		j = 1, \dots, n,\\
		t_{0} &= 
		\left\{
		\begin{array}{cl}
			\mu_{l} - \bar{\tau} \bar{\kappa} & \text{ if } \bar{\tau} \bar{\kappa} < \mu_{l}\\
			0 & \text{ if } \bar{\tau} \bar{\kappa} \in [\mu_{l}, \mu_{u}]\\
			\mu_{u} - \bar{\tau} \bar{\kappa} & \text{ if } \bar{\tau} \bar{\kappa} > \mu_{u}
		\end{array}
		\right.
		,
	\end{align}
	where $\mu_{l} = \gamma \mu \beta$ and $\mu_{u} = \gamma \mu \beta^{-1}$, for a fixed $0 < \beta \leq 1$.
	Then, define
	\begin{align}
		v &= t - \dfrac{e^{T}t + t_{0}}{n+1} e,\\
		v_{0} &= t_{0} - \dfrac{e^{T}t + t_{0}}{n+1}.
	\end{align}
	A correction is obtained by solving the linear system
	\begin{align}
		\label{eq:NS_cor2_dual}
		\revision{red}{-\prho \delta^{c}_{x}} + A^{T} \delta^{c}_{y} + \delta^{c}_{s} - c \delta^{c}_{\tau}			&= 0,\\
		\label{eq:NS_cor2_primal}
		 A \delta^{c}_{x} \revision{red}{+ \drho \delta^{c}_{y}} -b \delta^{c}_{\tau} 							&= 0,\\
		\label{eq:NS_cor2_opt}
	    -c^{T} \delta^{c}_{x} + b^{T} \delta^{c}_{y} \revision{red}{+ \grho \delta^{c}_{\tau}} -\delta^{c}_{\kappa}   	&= 0,\\
	    \label{eq:NS_cor2_comp_xs}
	    S \delta^{c}_{x} + X \delta^{c}_{s}                   			&= v,\\
	    \label{eq:NS_cor2_comp_tk}
		\kappa \delta^{c}_{\tau} + \tau \delta^{c}_{\kappa}   			&= v_{0},
	\end{align}
	which yields a corrected search direction
	\begin{align*}
		(
			\delta_{x},
			\delta_{y},
			\delta_{s},
			\delta_{\tau},
			\delta_{\kappa}
		) +
		(
			\delta^{c}_{x},
			\delta^{c}_{y},
			\delta^{c}_{s},
			\delta^{c}_{\tau},
			\delta^{c}_{\kappa}
		).
	\end{align*}
	The corrected direction is accepted if it results in an increased step size.
	
	Finally, additional centrality corrections are computed only if a sufficient increase in the step size is observed.
	Specifically, as suggested in \cite{Andersen2000}, an additional correction is computed only if the new step size $\alpha$ satisfies
	\begin{align}
		\label{eq:test_keep_correcting}
		\alpha \geq 1.10 \times \alpha^{\text{max}}.
	\end{align}

\revision{red}{
\subsubsection{Regularizations}

    Following \cite{Friedlander2012}, the regularizations are updated as follows.
    Let $\prho^{k}, \drho^{k}, \grho^{k}$ denote the regularization terms at iteration $k$.
    We set $\prho^{0} = \drho^{0} = \grho^{0} = 1$, and use the update rule
    \begin{align}
        \prho^{k+1} = \max \left( \sqrt{\epsilon}, \frac{\prho^{k}}{10} \right),\\ 
        \drho^{k+1} = \max \left( \sqrt{\epsilon}, \frac{\drho^{k}}{10} \right),\\
        \grho^{k+1} = \max \left( \sqrt{\epsilon}, \frac{\grho^{k}}{10} \right),
    \end{align}
    where $\epsilon$ denotes the machine precision, e.g., $\epsilon \simeq 10^{-16}$ for double-precision floating point arithmetic.
    
    Further details on the role of regularizations in the resolution of the Newton system are given in Section \ref{sec:KKT}.
    Let us only mention here that $\prho, \drho, \grho$ may become too small to ensure that the Newton system is properly regularized, e.g., for badly scaled problems. 
    When this is the case, we increase the regularizations by a factor of $100$, and terminate the algorithm if three consecutive increases fail to resolve the numerical issues.
}

\subsubsection{Step size}
	
	Once the final search direction has been computed, the step size $\alpha$ is given by
	\begin{align}
		\alpha = 0.9995 \times \alpha^{max},
	\end{align}
	where
	\begin{align*}
		 \alpha^{max} = \max \left\{ 0 \leq \alpha \leq 1 \mid (x, s, \tau, \kappa) + \alpha (\delta_{x}, \delta_{s}, \delta_{\tau}, \delta_{\kappa}) \geq 0 \right\}.
	\end{align*}

\subsubsection{Stopping criteria}
\label{sec:ipm:stoppingCriteria}

	The algorithm stops when, up to numerical tolerances, one of the following three cases holds: the current iterate is optimal, the primal problem is proven infeasible, the dual problem is proven infeasible (unbounded primal).
	
	The problem is declared solved to optimality if
	\begin{align}
		\label{eq:stop_primalfeas}
		\dfrac{\normInf{r_{p}}}{\tau (1 + \normInf{b})} &< \varepsilon_{p},\\
		\label{eq:stop_dualfeas}
		\dfrac{\normInf{r_{d}}}{\tau (1 + \normInf{c})} &< \varepsilon_{d},\\
		\label{eq:stop_opt}
		\dfrac{|c^{T}x - b^{T}y|}{\tau + |b^{T}y|} &< \varepsilon_{g},
	\end{align}
	where $\varepsilon_{p}, \varepsilon_{d}, \varepsilon_{g}$ are positive parameters.
	The above criteria are independent of the magnitude of $\tau$, and correspond to primal feasibility, dual feasibility and optimality, respectively.
	
	Primal or dual infeasibility is detected if
	\begin{align}
		\label{eq:stop_inf_mu}
		\mu &< \varepsilon_{i},\\
		\label{eq:stop_inf_tau}
		\frac{\tau}{\kappa} &< \varepsilon_{i},
	\end{align}
	where $\varepsilon_{i}$ is a positive parameter.
	When this is the case, a \revision{blue}{complementary} solution with small $\tau$ has been found.
	If $c^{T}x < - \varepsilon_{i}$, the problem is declared dual infeasible (primal unbounded), and $x$ is an unbounded ray.
	If $-b^{T}y < - \varepsilon_{i}$, the problem is declared primal infeasible (dual unbounded), and $y$ is a Farkas dual ray.
	
	Finally, premature termination criteria such as numerical instability, time limit or iteration limit are discussed in Section \ref{sec:implementation}.

\revision{red}{
\section{Solving linear systems}
\label{sec:KKT}
}

Search directions and centrality corrections are obtained by solving several Newton systems such as 	\eqref{eq:NS_aff_dual}-\eqref{eq:NS_aff_comp_tk}, all with identical left-hand side matrix but different right-hand side.	
\revision{red}{Specifically,} each Newton system has the form
\begin{align}
	\label{eq:NewtonSystem}
	\left[
	\begin{array}{ccccccccccc}
		\revision{red}{-\prho I} 		&& A^{T} 	&& I		&&	-c		&&		\\
		A 		&& \revision{red}{\drho I} 			&&		&&	-b		&&		\\
		-c^{T} 	&& b^{T}		&&		&& \revision{red}{\grho}			&& -1 	\\
		S 		&&  			&& X		&&			&&		\\
		 		&&  			&&		&& \kappa	&& \tau	
	\end{array}
	\right]
	\left[
	\begin{array}{l}
		\delta_{x}\\
		\delta_{y}\\
		\delta_{s}\\
		\delta_{\tau}\\
		\delta_{\kappa}
	\end{array}
	\right]
	=
	\left[
		\begin{array}{l}
			\xi_{d}\\
			\xi_{p}\\
			\xi_{g}\\
			\xi_{xs}\\
			\xi_{\tau \kappa}
		\end{array}
	\right]
	,
\end{align}
where $\xi_{p}, \xi_{d}, \xi_{g}, \xi_{xs}, \xi_{\tau \kappa}$ are \revision{red}{appropriate} right-hand side vectors.
\revision{red}{The purpose of this section is to provide further details on the techniques used for the resolution of \eqref{eq:NewtonSystem}, and their implementation in Tulip.
}

\revision{red}{
\subsection{Augmented system}
}

    First, we eliminate $\delta_{s}$ and $\delta_{\kappa}$ as follows:
    \begin{align}
    	\delta_{s} 		&= X^{-1}(\xi_{xs} - S\delta_{x}),\\
    	\delta_{\kappa} 	&= \tau^{-1}(\xi_{\tau \kappa} - \kappa \delta_{\tau}),
    \end{align}
    which yields
    \begin{align}
    	\label{eq:NewtonSystem_reduced}
    	\left[
    	\begin{array}{ccccccccccc}
    		\revision{red}{-(\Theta^{-1} + \prho I)}		&& A^{T} 	&&	-c		\\
    		A 			&&  \revision{red}{\drho I}			&&	-b		\\
    		-c^{T} 		&& b^{T}		&&	\tau^{-1} \kappa \revision{red}{ + \grho}		
    	\end{array}
    	\right]
    	\left[
    	\begin{array}{l}
    		\delta_{x}\\
    		\delta_{y}\\
    		\delta_{\tau}
    	\end{array}
    	\right]
    	=
    	\left[
    		\begin{array}{l}
    			\xi_{d} - X^{-1} \xi_{xs}\\
    			\xi_{p}\\
    			\xi_{g} +\tau^{-1} \xi_{\tau \kappa}
    		\end{array}
    	\right]
    	,
    \end{align}
    where $\Theta = XS^{-1}$.
    
    \revision{red}{As outlined in \cite{Andersen2000,Wright1997}}, a solution to \eqref{eq:NewtonSystem_reduced} is obtained by first solving
    \begin{align}
    	\label{eq:augsys_hsd}
    	\left[
    	\begin{array}{ccccccccccc}
    		\revision{red}{-(\Theta^{-1} + \prho I)}		&& A^{T} 	\\
    		A 			&& \revision{red}{\drho I}
    	\end{array}
    	\right]
    	\left[
    	\begin{array}{l}
    		p\\
    		q
    	\end{array}
    	\right]
    	&=
    	\left[
    		\begin{array}{l}
    			c\\
    			b
    		\end{array}
    	\right]
    	,
    \end{align}
    and
    \begin{align}
    	\label{eq:augsys_hsd_bis}
    	\left[
    	\begin{array}{ccccccccccc}
    		\revision{red}{-(\Theta^{-1} + \prho I)}		&& A^{T} 	\\
    		A 			&& \revision{red}{\drho I}
    	\end{array}
    	\right]
    	\left[
    	\begin{array}{l}
    		u\\
    		v
    	\end{array}
    	\right]
    	&=
    	\left[
    		\begin{array}{l}
    			\xi_{d} - X^{-1} \xi_{xs}\\
    			\xi_{p}
    		\end{array}
    	\right]
    	.
    \end{align}
    \revision{red}{Linear systems of the form \eqref{eq:augsys_hsd} and \eqref{eq:augsys_hsd_bis} are referred to as \emph{augmented systems}.}
    Then, $\delta_{x}, \delta_{y}, \delta_{\tau}$ are computed as follows:
    \begin{align}
    	\delta_{\tau} 	&= \frac{\xi_{g} + \tau^{-1}\xi_{\tau \kappa} + c^{T}u + b^{T}v}{\tau^{-1} \kappa \revision{red}{+ \grho} - c^{T}p + b^{T}q},\\
    	\delta_{x} 		&= u + \delta_{\tau} p,\\
    	\delta_{y} 		&= v + \delta_{\tau} q.
    \end{align}
    \revision{red}{Note that \eqref{eq:augsys_hsd} does not depend on the right-hand side $\xi$.
    Thus, it is only solved once per IPM iteration, and its solution is reused when solving subsequent Newton systems.
    
    Finally, as pointed in \cite{Friedlander2012}, the augmented system's structure motivates the following observations.
    First, the use of primal-dual regularizations controls the effective condition number of the augmented system, which, in turn, improves the algorithm's numerical behavior.
    Second, the augmented system's matrix is symmetric quasi-definite.
    This allows the use of efficient symmetric indefinite factorization techniques, which only require one symbolic analysis at the beginning of the optimization.
    In particular, dual regularizations ensure that this quasi-definite property is retained even when $A$ does not have full rank.
    Third, directly solving the augmented system implicitly handles dense columns in $A$, which make the system of normal equations dense \cite{Wright1997}.
    We have also found this approach to be more numerically stable than a normal equations system-based approach.
    }

\revision{red}{
\subsection{Black-box linear solvers}
    
    The augmented system may be solved using a number of techniques, with direct methods --namely, symmetric factorization techniques-- being the most popular choice.
    Importantly, the algorithm itself is unaffected by \emph{how} the augmented system is solved, provided that it is solved accurately.
}
    Our implementation leverages Julia's multiple dispatch feature and built-in support for linear algebra,
	thus allowing to disentangle the algorithmic framework from the linear algebra implementation.
	
	First, the interior-point algorithm is defined over abstract linear algebra structures.
	Namely, the constraint matrix $A$ is \revision{red}{treated} as an \texttt{AbstractMatrix}, whose concrete type is only known once the model is instantiated.
	Julia's standard library includes extensive support for linear algebra, thus removing the need for a custom abstract linear algebra layer.
	
	\revision{red}{
        Second, while the reduction from the Newton system to the augmented system is performed explicitly, the latter is solved by a black-box linear solver.
        Specifically, we design an \texttt{AbstractKKTSolver} type, from which concrete linear solver implementations inherit.
        The \texttt{AbstractKKTSolver} interface is deliberately minimal, and consists of three functions:\footnote{In Julia, a \texttt{!} is appended to functions that mutate their arguments.} \texttt{setup}, \texttt{update!}, and \texttt{solve!}.
        
        A linear solver is instantiated at the beginning of the optimization using the \texttt{setup} function.
        Custom options can be passed to \texttt{setup} so that the user can select a linear solver of their choice.
        At the beginning of each IPM iteration, the linear solver's state is updated by calling the \texttt{update!} function.
        For instance, if a direct method is used, this step corresponds to updating the factorization.
        Following the call to \texttt{update!}, augmented systems can be solved through the \texttt{solve!} function.
        Default, generic, linear solvers are described in Section \ref{sec:implementation:linalg}, and an example of specialized linear solver is given in Section \ref{sec:res:colgen}.
        Specific details are provided in Tulip's online documentation.\footnote{\url{https://ds4dm.github.io/Tulip.jl/dev/}}
    }
	
	\revision{red}{Finally}, specialized methods are automatically dispatched based on the (dynamic) type of $A$.
	These include matrix-vector and matrix-matrix product, as well as matrix factorization routines.
	We emphasize that the dispatch feature is a core component of the Julia programming language, and is therefore entirely transparent to the user.
	Consequently, one can easily define custom routines that exploit certain properties of $A$, so as to speed-up computation or reduce memory overheads.
	Furthermore, this customization is entirely independent of the interior-point algorithm, thus allowing to properly assess the impact of different linear algebra implementations.

\revision{blue}{
\section{Presolve}
\label{sec:presolve}

    Tulip's presolve module performs elementary reductions, all of which are described in \cite{Andersen1995} and \cite{Gondzio1997_presolve}.
    Therefore, in this section, we only outline the presolve procedure; further implementation details are given in Section \ref{sec:implementation}.
    
    \subsection{Presolve}
    
    We only perform reductions that do not introduce any additional non-zero coefficients, i.e., fill-in, to the problem.
    The presolve procedure is outlined in Algorithm \ref{alg:presolve}, and proceeds as follows.
    
    First, we ensure all bounds are consistent, remove all empty rows and columns, and identify all row singletons, i.e., rows that contain with a single non-zero coefficient.
    Then, a series of passes is performed until no further reduction is possible.
    At each pass, the following reductions are applied: empty rows and columns, fixed variables, row singletons, free and implied free column singletons, forcing and dominated rows, and dominated columns.
    The presolve terminates if infeasibility or unboundedness is detected, in which case an appropriate primal or dual ray is constructed.
    If all rows and columns are eliminated, the problem is declared solved, and a primal-dual optimal solution is constructed.
    
    Finally, to improve the numerical properties of the problem, rows and columns are re-scaled as follows:
    \begin{align}
        \label{eq:presolve:scaling}
        \tilde{A} = D^{(r)} \times A \times D^{(c)},
    \end{align}
    where $\tilde{A}$ is the scaled matrix, $A$ is the constraint matrix of the reduced problem,  and $D^{(r)}$, $D^{(c)}$ are diagonal matrices with coefficients
    \begin{align}
        D^{(r)}_{i} &= \frac{1}{\sqrt{\norm{A_{i, \cdot}}}}, \ \ \ \forall i,\\
        D^{(c)}_{j} &= \frac{1}{\sqrt{\norm{A_{\cdot, j}}}}, \ \ \ \forall j.
    \end{align}
    Column and row bounds, as well as the objective, are scaled appropriately.
        
    \begin{algorithm}
		\begin{algorithmic}
			\REQUIRE Initial LP
			
			\vspace{5pt}
			\STATE Remove empty rows
			\STATE Remove empty columns
			
			\vspace{5pt}
			\REPEAT
				
				\vspace{5pt}
				\STATE Check for bounds inconsistencies
				\STATE Remove empty columns
				
				\vspace{5pt}
				\STATE Remove row singletons
				\STATE Remove fixed variables
				
				\vspace{5pt}
				\STATE Remove row singletons
				\STATE Remove forcing/dominated rows
				
				\vspace{5pt}
				\STATE Remove row singletons
				\STATE Remove free columns singletons
				
				\vspace{5pt}
				\STATE Remove row singletons
				\STATE Remove dominated columns
			
			\vspace{5pt}
			\UNTIL{No reduction is found}
			
			\vspace{5pt}
			\STATE Scale rows and columns
		\end{algorithmic}
		\caption{Presolve procedure}
		\label{alg:presolve}
	\end{algorithm}
	
	\subsection{Postsolve}
	
    A primal-dual solution to the presolved problem is computed using the interior-point algorithm described in Section \ref{sec:ipm}.
    A solution to the original problem is then constructed in a postsolve phase, whose algorithmic details are detailed in \cite{Andersen1995,Gondzio1997_presolve}.
    Note that, in general, the postsolve solution is \emph{not} an interior point with respect to the original problem, e.g., some variables may be at their upper or lower bound.
}

\section{Implementation details}
\label{sec:implementation}

\revision{blue}{Tulip is an officially registered Julia package}, and is publicly available\footnote{Source code is available at \url{https://github.com/ds4dm/Tulip.jl}, and online documentation at \url{https://ds4dm.github.io/Tulip.jl/dev/}} under an open-source license.
\revision{blue}{The entire source code comprises just over $4,000$ lines of Julia code, which makes it easy to read and to modify.
}
The code is single-threaded, however external linear algebra libraries may exploit multiple threads.

\revision{blue}{
    We provide an interface to \texttt{MathOptInterface} \cite{Legat2020_MOI}, a solver-agnostic abstraction layer for optimization.
    Thus, Tulip is readily available through both \texttt{JuMP} \cite{Dunning2017_JuMP}, an open-source algebraic modeling language embedded in Julia, and the convex optimization modeling framework \texttt{Convex} \cite{Convex.jl-2014}.

    Finally, Tulip supports arbitrary precision arithmetic, thus allowing, for instance, to solve problems in quadruple (128 bits) precision.
    This functionality is available from Tulip's direct API and through the \texttt{MathOptInterface} API; it is illustrated in Section \ref{sec:res:precision}.
}

\subsection{Bounds on variables}
\label{sec:implementation:bounds}

\revision{blue}{
    Tulip stores LP problems in the form
    \begin{align}
        \label{eq:LP_gen}
        \begin{array}{rrcll}
        (LP) \ \ \ 
        \displaystyle \min_{x} \ \ \ 
        && c^{T}x & + \ c_{0}\\
        s.t. \ \ \ 
        & l^{b}_{i} \leq & \sum_{j} a_{i, j} x_{j} & \leq u^{b}_{i}, & \ \ \ \forall i = 1, ..., m,\\
        & l^{x}_{j} \leq & x_{j} & \leq u^{x}_{j}, & \ \ \ \forall j = 1, ..., n,\\
        \end{array}
    \end{align}
    where $l^{b,x}_{i, j}, u^{b, x}_{i, j} \in \mathbb{R} \cup \{ - \infty, + \infty \}$, i.e., some bounds may be infinite.
    Before being passed to the interior-point optimizer, the problem is transformed into standard form.
    This transformation occurs after the presolve phase, and is transparent to the user.
    In particular, primal-dual solutions are returned with respect to formulation \eqref{eq:LP_gen}.
}

	Free variables are an outstanding issue for interior-point methods, see, e.g. \cite{Wright1997, Anjos2008}, and are not supported explicitly in Tulip.
	Instead, free variables are automatically split into the difference of two non-negative variables, with the knowledge that this reformulation may introduce some numerical instability.
	
	Although finite upper bounds may be treated as arbitrary constraints, it is more efficient to handle them separately.
	Let $\mathcal{I}$ denote the set of indices of upper-bounded variables.
	Upper-bound constraints then write 
	\begin{align}
		x_{i} \leq u_{i}, \ \ \forall i \in \mathcal{I},
	\end{align}
	which we write in compact form $Ux \leq u$, where $U \in \mathbb{R}^{|\mathcal{I}| \times n}$ and
	\begin{align*}
		U_{i, j} = 
		\left\{
			\begin{array}{ll}
				1 & \text{if } i = j \in \mathcal{I}\\
				0 & \text{otherwise}
			\end{array}
		\right.
		.
	\end{align*}
	Therefore, \revision{blue}{internally}, Tulip \revision{blue}{solves} linear programs of the form
	\begin{align}
		\label{eq:standard_LP_bounds}
		\begin{array}{rll}
			(P) \ \ \ \displaystyle \min_{x, w} \ \ \ 
			& c^{T} x\\
			s.t. \ \ \ 
			& A x = b,\\
			& U x + w =u\\
			& x, w  \geq 0,
		\end{array}
		&
		\hspace*{1cm}
		\begin{array}{rll}
			(D) \ \ \ \displaystyle \max_{y, s, z} \ \ \
			& b^{T} y - u^{T}z \ \ \ \ \ \  \\
			s.t. \ \ 
			& A^{T}y + s - U^{T}z= c,\\
			& s, z \geq 0.\\
			\ 
		\end{array}
	\end{align}
	Let us emphasize that handling upper bounds separately only affects the underlying linear algebra operations, not the interior-point algorithm.
	
	The Newton system \eqref{eq:NewtonSystem} then writes
	\begin{align}
		\label{eq:NS_bounds}
		\left[
		\begin{array}{ccccccccccccccc}
			\revision{red}{-\prho I} 		&& 		&& A^{T} 	&& I		&& -U^{T}	&&	-c		&&		\\
			A 		&& 		&&  \revision{red}{\drho I}			&&		&& 			&&	-b		&&		\\
			U 		&& I		&&  			&&		&& 			&&	-u		&&		\\
			-c^{T} 	&& 		&& b^{T}		&&		&& -u^{T}	&& \revision{red}{\grho}			&& -1 	\\
			S 		&& 		&&  			&& X		&& 			&&			&&		\\
			 		&& Z		&&  			&& 		&& W 		&&			&&		\\
			 		&& 		&&  			&&		&& 			&& \kappa	&& \tau	
		\end{array}
		\right]
		\left[
		\begin{array}{l}
			\delta_{x}\\
			\delta_{w}\\
			\delta_{y}\\
			\delta_{s}\\
			\delta_{z}\\
			\delta_{\tau}\\
			\delta_{\kappa}
		\end{array}
		\right]
		=
		\left[
			\begin{array}{l}
				\xi_{d}\\
				\xi_{p}\\
				\xi_{u}\\
				\xi_{g}\\
				\xi_{xs}\\
				\xi_{wz}\\
				\xi_{\tau \kappa}
			\end{array}
		\right]
		,
	\end{align}
	and it reduces, after performing diagonal substitutions, to solving two augmented systems of the form
	\begin{align}
		\label{eq:augsys_hsd_bounds_reduced}
		\left[
		\begin{array}{cccccc}
			\revision{red}{-(\tilde{\Theta}^{-1} + \prho I)}		&& A^{T} \\
			A 			&& \revision{red}{\drho I} 		
		\end{array}
		\right]
		\left[
		\begin{array}{l}
			p\\
			q
		\end{array}
		\right]
		&=
		\left[
		\begin{array}{l}
			\tilde{\xi}_{d}\\
			\tilde{\xi}_{p}
		\end{array}
		\right]
		,
	\end{align}
	where $\tilde{\Theta} = \left( X^{-1}S + U^{T}(W^{-1}Z)U \right)^{-1}$.
	Note that $\tilde{\Theta}$ is a diagonal matrix with positive diagonal.
	Therefore, system \eqref{eq:augsys_hsd_bounds_reduced} has the same size and structure as \eqref{eq:augsys_hsd}.
	Furthermore, $\tilde{\Theta}$ can be computed efficiently using only vector operations, i.e., without any matrix-matrix nor matrix-vector product.
	
\subsection{Solver parameters}
\label{sec:implementation:params}

	The default values for numerical tolerances of Section \ref{sec:ipm:stoppingCriteria} are
	\begin{align*}
		\varepsilon_{p} = \revision{blue}{\sqrt{\epsilon}},\\
		\varepsilon_{d} = \revision{blue}{\sqrt{\epsilon}},\\
		\varepsilon_{g} = \revision{blue}{\sqrt{\epsilon}},\\
		\varepsilon_{i} = \revision{blue}{\sqrt{\epsilon}},
	\end{align*}
	\revision{blue}{where $\epsilon$ is the machine precision, which depends on the arithmetic.
	For instance, double precision (64 bits) floating point arithmetic corresponds to $\epsilon_{64} \simeq 10^{-16}$, while quadruple precision (128 bits) corresponds to $\epsilon_{128} \simeq 10^{-34}$.}
	
	When computing additional centrality corrections, we use the following default values:
	\begin{align*}
		\revision{blue}{\gamma_{min}} = 10^{-1},\\
		\beta = 10^{-1}.
	\end{align*}
	The default maximum number of centrality corrections is set to $5$.
	
	Finally, the maximum number of IPM iterations is set to a default of $100$.
	A time limit may be imposed by the user, in which case it is checked at the beginning of each IPM iteration.
	
\revision{red}{
\subsection{Default linear solvers}
\label{sec:implementation:linalg}
}
    \revision{red}{Several generic linear algebra implementations are readily available in Tulip, and can be selected without requiring any additional implementation.
    
    The default settings are as follows.
    First, $A$ is stored in a \texttt{SparseMatrixCSC} struct, i.e., in compressed sparse column format.
    Elementary linear algebra operations, e.g., matrix-vector products, employ Julia's standard library \texttt{SparseArrays}.
    Augmented systems are then solved by a direct method, namely, an $LDL^{T}$ factorization of the quasi-definite augmented system.
    Sparse factorizations use either the CHOLMOD module of SuiteSparse \cite{SuiteSparse}, or the \texttt{LDLFactorizations} package \cite{LDLFactorizations}, a Julia translation of SuiteSparse's $LDL^{T}$ factorization code that supports arbitrary arithmetic.
    Tulip uses the former for double precision floating point arithmetic, and the latter otherwise.
    Finally, the solver's log indicates: the model's arithmetic, the linear solver's backend, e.g., CHOLMOD, and the linear system being solved, i.e., either the augmented system of the normal equations system.
    
    As mentioned in Section \ref{sec:KKT}, custom options for linear algebra can be passed to the solver.
    Specifically, the \texttt{MatrixOptions} parameter lets the user select a matrix implementation of their choice, and the \texttt{KKTOptions} parameter is used to specify a choice of linear solver.
    Their usage is depicted in Figure \ref{fig:KKTOptions}.
    
    In Figure \ref{fig:KKTOptions:default}, the default settings are used.
    The model is instantiated at line 3; the \texttt{Model\{Float64\}} syntax indicates that \texttt{Float64} arithmetic is used.
    Then, the problem is read from the \texttt{problem.mps} file at line 4, and the model is solved at line 6.
    Figures \ref{fig:KKTOptions:Dense}, \ref{fig:KKTOptions:CholmodPD}, \ref{fig:KKTOptions:LDLFact} are identical, but select different linear algebra implementations by setting the appropriate \texttt{MatrixOptions} and \texttt{KKTOptions} parameters.
    
    Figure \ref{fig:KKTOptions:Dense} illustrates the use of dense linear algebra.
    Line 7 indicates that $A$ should be stored as a dense matrix.
    Then, at line 8, a dense linear solver is selected through the \texttt{SolverOptions(Dense\_SymPosDef)} setting.
    In this case, the augmented system is reduced to the (dense) normal equations systems, and a dense Cholesky factorization is applied; BLAS/LAPACK routines are automatically called when using single and double precision floating point arithmetic, otherwise Julia's generic routines are called.
    
    In the example of Figure \ref{fig:KKTOptions:CholmodPD}, linear systems are reduced to the normal equations system, and CHOLMOD's sparse Cholesky factorization is applied.
    Note that a single dense column in $A$ results in a fully dense normal equations systems.
    Thus, in the absence of a mechanism for handling dense columns, this approach may be impractical for some large problems.
    Finally, in Figure \ref{fig:KKTOptions:LDLFact}, the augmented system is solved using an $LDL^{T}$ factorization, computed by \texttt{LDLFactorizations}.
}

    \begin{figure}
        \centering
        \begin{subfigure}{\textwidth}
\begin{lstlisting}[language=Python]
import Tulip

model = Tulip.Model{Float64}()             # Instantiate model
Tulip.load_problem!(model, "problem.mps")  # Read problem

Tulip.optimize!(model)                     # Solve the problem
\end{lstlisting}
        \caption{Sample code using default linear algebra settings}
        \label{fig:KKTOptions:default}
        \end{subfigure}
        \par \bigskip
        \begin{subfigure}{\textwidth}
        \begin{lstlisting}[language=Python]
import Tulip

model = Tulip.Model{Float64}()             # Instantiate model
Tulip.load_problem!(model, "problem.mps")  # Read problem

# Select dense linear algebra
model.params.MatrixOptions = MatrixOptions(Matrix)
model.params.KKTOptions = SolverOptions(Dense_SymPosDef)

Tulip.optimize!(model)                     # Solve the problem
\end{lstlisting}
        \caption{Sample code using dense linear algebra}
        \label{fig:KKTOptions:Dense}
        \end{subfigure}
        \par \bigskip
        \begin{subfigure}{\textwidth}
        \begin{lstlisting}[language=Python]
import Tulip

model = Tulip.Model{Float64}()             # Instantiate model
Tulip.load_problem!(model, "problem.mps")  # Read problem

# Solve the normal equations with CHOLMOD
model.params.KKTOptions = SolverOptions(CholmodSolver, normal_equations=true)

Tulip.optimize!(model)                     # Solve the problem
\end{lstlisting}
        \caption{Sample code using CHOLMOD to solve the normal equations system}
        \label{fig:KKTOptions:CholmodPD}
        \end{subfigure}
        \par \bigskip
        \begin{subfigure}{\textwidth}
        \begin{lstlisting}[language=Python]
import Tulip

model = Tulip.Model{Float64}()             # Instantiate model
Tulip.load_problem!(model, "problem.mps")  # Read problem

# Solve the augmented system with LDLFactorizations
model.params.KKTOptions = SolverOptions(LDLFact_SymQuasDef)

Tulip.optimize!(model)                     # Solve the problem
\end{lstlisting}
        \caption{Sample code using \texttt{LDLFactorizations}}
        \label{fig:KKTOptions:LDLFact}
        \end{subfigure}
        
        \caption{Code examples for reading and solving a problem with various linear algebra implementations.}
        \label{fig:KKTOptions}
    \end{figure}

\section{Computational results}
\label{sec:res}

\revision{blue}{%
In this section, we compare Tulip to several open-source and commercial solvers, focusing on those that are available to Julia users.
Let us emphasize that our goal is \emph{not} to perform a comprehensive benchmark of interior-point LP solvers.

We evaluate Tulip's performance and robustness in the following three settings.
First, in Section \ref{sec:res:plato}, we consider general LP instances from H. Mittelmann's benchmark,\footnote{\url{http://plato.asu.edu/ftp/lpbar.html}} which are solved using generic sparse linear algebra.
Then, in Section \ref{sec:res:colgen}, we consider structured instances that arise in decomposition methods, for which we develop specialized linear algebra.
Finally, in Section \ref{sec:res:precision}, we illustrate Tulip's ability to use different levels of arithmetic precision by solving problems in higher precision.
}

%
%
\subsection{\revision{blue}{ Results on general LP instances}}
\label{sec:res:plato}
	
	\revision{blue}{%
	We select all instances from H. Mittelmann's benchmark of barrier LP solvers, except \texttt{qap15} and \texttt{L1\_sixm1000obs}.
	The former is identical to \texttt{nug15}, and the latter could not be solved by any solvers in the prescribed time limit.
	This yields a testset of 43 medium to large-scale instances.}
	We compare the following open-source and commercial solvers: Clp $1.17$ \cite{CLP}, GLPK 4.64 \cite{GLPK}, ECOS 2.0\cite{Domahidi2013}, Tulip 0.5.0, CPLEX 12.10 \cite{CPLEX}, Gurobi 9.0 \cite{Gurobi} and Mosek 9.2 \cite{Mosek}.
	All are accessed through their respective Julia interface.
	We run the interior-point algorithm of each solver with a single thread, no crossover, and a \revision{blue}{$10,000$s} time limit.
	\revision{blue}{For Tulip, the maximum number of IPM iterations is increased from the default 100 to 500.}
	All other parameters are left to their default values. 
	
	Experiments are carried out on \revision{blue}{a cluster of machines equipped with dual Intel Xeon 6148-2.4GHz CPUs, and varying amounts of RAM. Each job is run with a single thread and 16GB of memory.}
	Scripts for running these experiments are available online,\footnote\revision{blue}{\url{https://github.com/mtanneau/LPBenchmarks}}
	\revision{blue}{together with the logfiles of each solver.}

	Computational results are displayed in Table \ref{tab:res:plato}.
	\revision{blue}{%
	For each solver, we report the total number of instances solved, the mean runtime, and individual runtimes for each instance.
	Segmentation faults are indicated by \texttt{seg}, timeouts by \texttt{t}, other failures by \texttt{f}, and reduced accuracy solutions by \texttt{r}.
	The time to read in the data is not included.
}
\revision{red}{
	Mean runtimes are shifted geometric means
	\begin{align*}
	    \mu_{\delta}(t_{1}, ..., t_{N}) 
	    = \left( \prod_{i=1}^{N} (t_{i} + \delta) \right)^{\frac{1}{N}} - \delta
	    = \exp \left[ \frac{1}{N} \sum_{i=1}^{N} \log (t_{i} + \delta) \right] - \delta,
	\end{align*}
    with $\delta=10$ seconds.
}
    
    \revision{blue}{
\begin{table}
    \centering
    \caption{Results on the Mittelmann test set}
    \label{tab:res:plato}
    \begin{tabular}{lrrrrrrr}
\toprule
Problem          &    Clp &  CPLEX &   ECOS &   GLPK & Gurobi &  Mosek &  Tulip\\
\midrule
Solved           &     25 &     39 &     26 &      6 &     41 &     43 &     33\\
Average          & 1607.5 &   52.4 & 2344.7 & 7092.8 &   42.5 &   32.0 &  604.6\\
\midrule
L1\_sixm250obs   & \texttt{seg} &   23.7 & \texttt{f} & \texttt{f} & \texttt{f} &  148.8 & \texttt{f}\\
Linf\_520c       & \texttt{seg} &   25.9 & \texttt{f} & \texttt{seg} &   30.4 &   22.7 & \texttt{t}\\
brazil3          &    2.3 &    0.2 & \texttt{f} & \texttt{f} &    0.6 &    0.6 &    2.3\\
buildingenergy   & \texttt{f} &   18.3 &  362.0 & \texttt{f} &   19.9 &   16.6 &   39.1\\
chrom1024-7      & \texttt{seg} &    0.3 &   96.8 & \texttt{f} &    1.5 &    3.2 &    3.6\\
cont1            & 2322.4 &    5.2 &  138.0 & \texttt{f} &    5.9 &   12.6 &   26.9\\
cont11           &  626.8 & \texttt{f} &  174.7 & \texttt{f} &   14.8 &   12.7 &   51.4\\
dbic1            &  118.5 &    9.6 &  332.9 & \texttt{f} &    9.2 &    9.4 &   30.2\\
degme            & \texttt{t} &  235.5 & \texttt{f} & \texttt{f} &  308.3 &  253.0 & \texttt{t}\\
ds-big           &  237.1 &   29.6 &  331.6\texttt{r} & \texttt{f} &   31.1 &   16.5 &   95.2\\
ex10             & \texttt{t} &    6.3 & \texttt{f} & \texttt{f} &   46.8 &   21.6 & 5427.5\\
fome13           &  413.3 &   19.0 & 1284.0 & \texttt{f} &   17.8 &   16.9 &  538.2\\
irish-e          &  363.6 &   21.1 & \texttt{f} & \texttt{f} &   16.5 &   20.3 &   35.4\\
karted           & 6509.1 &   89.1 & 5801.8 & 9334.9 &  115.7 &   44.4 & 3786.9\\
neos             &  433.4 &   26.7 & 1201.0\texttt{r} & \texttt{seg} &   39.5 &   33.1 &  419.5\\
neos1            & \texttt{f} &    4.9 &  167.2 & \texttt{f} &    6.3 &    4.1 &  130.4\\
neos2            & \texttt{f} &    4.0 &  129.6 & \texttt{f} &    4.7 &    3.4 &  462.1\\
neos3            & \texttt{seg} &   26.5 & 1282.9 & \texttt{f} &   33.7 &   17.1 & 1358.2\\
neos5052403      & 2067.9 &   43.6 & 3489.4\texttt{r} & \texttt{f} &   28.1 &   13.1 &  475.4\\
ns1644855        & \texttt{t} &  334.2 & \texttt{f} & \texttt{f} &  437.0 &  468.0 & \texttt{t}\\
ns1687037        & \texttt{t} & \texttt{f} & \texttt{f} & \texttt{f} &   19.6 &   12.3 &  330.4\\
ns1688926        & \texttt{t} &   25.7 & \texttt{f} & \texttt{f} & \texttt{f} &    2.0 & \texttt{f}\\
nug08-3rd        & \texttt{t} &    3.4 & \texttt{f} & \texttt{f} &    2.8 &   55.0 & \texttt{f}\\
nug15            &  352.2 &   12.7 & 1871.4 & \texttt{f} &    0.9 &   17.8 &  998.6\\
pds-100          & \texttt{t} &  168.2 & \texttt{f} & \texttt{f} &  106.2 &  152.7 & \texttt{f}\\
pds-40           & 1303.7 &   25.6 & \texttt{f} & \texttt{f} &   22.1 &   32.9 & 5000.6\\
psched3-3        & \texttt{t} &   86.6 & \texttt{f} & \texttt{f} &  147.5\texttt{r} &  148.5 & \texttt{t}\\
rail02           & \texttt{t} &  208.9 & \texttt{f} & \texttt{f} &  134.3 &  195.8 & \texttt{t}\\
rail4284         & 5407.8 &   72.2 & 8578.3\texttt{r} & \texttt{f} &  133.5 &   81.0 & 1049.5\\
s100             & 1587.5 &   29.7 & \texttt{f} & \texttt{f} &   38.1 &   30.2 &  894.2\\
s250r10          &  263.5 &   17.8 & \texttt{f} & \texttt{f} &   25.4 &   30.8 &  257.2\\
savsched1        &  183.9 &   27.1 & 2355.4 & \texttt{f} &   25.9 &   55.5 &  138.8\\
self             &   21.5 &    3.2 &  162.3 &   45.1 &    4.0 &    3.1 &   13.3\\
shs1023          &  286.2 & \texttt{f} & \texttt{f} & \texttt{f} &   48.2 &   74.6 &  371.6\\
square41         &  202.8 &    3.4 & 1703.3\texttt{r} & \texttt{f} &    4.9 &   32.7 &  134.0\\
stat96v1         &  164.9 & \texttt{f} &  163.6\texttt{r} & \texttt{f} &   32.2\texttt{r} &    6.3 &   41.3\\
stat96v4         &    1.4 &    1.0 &   31.4 &  261.5 &    1.0 &    2.2 &    1.8\\
stormG2\_1000    & \texttt{seg} &   64.2 & 9593.0 & \texttt{f} &  122.6 &  131.9 &  216.3\\
stp3d            & 1864.1 &   31.6 & 1363.3 & \texttt{f} &   31.9 &   39.2 &  529.7\\
support10        & \texttt{f} &   17.1 & 6210.2 & \texttt{f} &   19.3 &   27.2 & 3553.1\\
tp-6             & 7872.1 &  150.9 & \texttt{f} & \texttt{f} &  203.8 &  312.3 & 5543.0\\
ts-palko         & 1757.1 &   35.9 & 1109.3 & 2315.7 &   50.7 &   31.7 &  841.1\\
watson\_2        &   65.0 &   24.4 & \texttt{f} &  111.0 &   26.6 &   30.1 & \texttt{f}\\
\bottomrule
    \end{tabular}\\
    \texttt{seg}: segmentation fault;
    \texttt{r}: reduced accuracy solution;
    \texttt{t}: time limit;
    \texttt{f}: other failure\\
    All times in seconds.
\end{table}
}
	
	\revision{blue}{
	    First, the three commercial solvers CPLEX, Gurobi and Mosek display similar performance and robustness, and outperform open-source alternatives by one to two orders of magnitude.
	    While CPLEX and Gurobi encountered numerical issues on a few instances, we found that these were resolved by activating crossover.
	    
	    Second, Clp displays a worse performance than expected, solving only $25$ problems with an average runtime about two times larger than Tulip's.
	    In fact, out of $43$ instances, we recorded $5$ segmentation faults, $8$ unidentified errors, with the $10,000$s time limit being reached on the remaining $10$ unsolved instances.
	    A more detailed analysis of the log suggests that segmentation faults and some unknown errors are caused by memory-related issues, i.e., large Cholesky factors that do not fit in memory.
	    We note that those errors do not occur when running Clp through its command-line executable: the executable performs additional checks to decide whether the model should be dualized; this can yield smaller linear systems and thus avoid memory issues. 
	    Nevertheless, given that the \texttt{dualize} option is not available in Clp's C interface\footnote{See discussion in \url{https://github.com/coin-or/Clp/issues/151}}, on which Clp's Julia wrapper is built, the present results best represent the behavior that Julia users would encounter.
	    
	    Third, among open-source solvers, Tulip is the top performer with $33$ instances solved and a mean runtime of $604.6$s, while GLPK has the worst performance with only $6$ instances reportedly solved.
	    Tulip's $5$ failures include $3$ instances that out of memory; for the remaining $2$, i.e., \texttt{ns1688926} and \texttt{watson\_2}, Tulip fails to reach the prescribed accuracy due to numerical issues.
	    A possible remedy to the latter will be discussed in Section \ref{sec:res:precision}.
	    Finally, out of the $26$ instances reported as solved by ECOS, $6$ were solved to reduced accuracy.
	    This situation typically corresponds to ECOS encountering numerical issues close to optimality, but a feasible or close-to-feasible solution is still available.
	}

%
%
\subsection{\revision{blue}{Results on structured LP instances}}
\label{sec:res:colgen}

    We now compare Tulip to state-of-the-art commercial solvers on a
\revision{blue}{%
    collection of structured problems, for which we design specialized linear algebra routines.
    Specifically, we consider the context of Dantzig-Wolfe (DW) decomposition \cite{Dantzig1960} in conjunction with a column-generation (CG) algorithm; we refer to \cite{Desaulniers2006column} for a thorough overview of DW decomposition and CG algorithms.
    Here, we focus on the resolution of the master problem, i.e., we consider problems of the form
	\begin{align}
		\label{eq:MP:objective}
		(MP) \ \ \ \min_{\lambda} \ \ \  & \sum_{r=1}^{R} \sum_{j=1}^{n_{r}} c_{r, j} \lambda_{r, j} + c_{0}^{T} \lambda_{0}\\	
		\label{eq:MP:convexity}
		s.t. \ \ \ 
		& \sum_{j=1}^{n_{r}} \lambda_{r, j} =1, \ \ \ r=1, ..., R,\\
		\label{eq:MP:linking}
		& \sum_{r=1}^{R} \sum_{j=1}^{n_{r}} a_{r, j} \lambda_{r, j} + A_{0} \lambda_{0} = b_{0},\\	
		\label{eq:MP:domain}
		& \lambda \geq 0,
	\end{align}
	where $R$ is the number of sub-problems, $m_{0}$ is the number of linking constraints, $n_{r}$ is the number of columns from sub-problem $r$, $A_{0} \in \R^{m_{0} \times n_{0}}$, and $\forall (r, j), a_{r, j} \in \R^{m_{0}}$.
	Let $M = R + m_{0}$ and $N = n_{0} + n_{1} + \dots + n_{R}$ be the number of constraints and variables in $(MP)$, respectively.
	In what follows, we focus on the case where (i) $R$ is large, typically in the thousands or tens of thousands, (ii) $m_{0}$ is not too large, typically in the hundreds, and (iii) the vectors $a_{r, j} \in \R^{m_{0}}$ and $A_{0}$ are dense.
}

\revision{blue}{
    \subsubsection{Instance collection}
    \label{sec:res:colgen:instances}
        
        We build a collection of master problems from two sources.
        First, we generate instances of Distributed Energy Resources (DER) coordination from  \cite{Anjos2019}.
        We select a renewable penetration rate $\xi =0.33$, a time horizon $T= \{24, 48, 96\}$, and a number of resources $R=\{1024, 2048, 4096, 8192, 16384, 32768\}$.
        Second, we select all two-stage stochastic programming (TSSP) problems from \cite{Gondzio2016_PDCGMlarge} that have at least $1,000$ scenarios.
        This yields $18$ DER instances, and $27$ TSSP instances.
        
        Then, each instance is solved by column generation; master problems are solved with Gurobi's barrier (with crossover) and sub-problems are solved with Gurobi's default settings.
        In the case of DER instances, which contain mixed-integer variables, only the root node of a branch-and-price tree is solved.
        Finally, at every tenth CG iteration and the last, the current master problem is saved.
        Thus, we obtain a dataset of $153$ master problems of varying sizes.
        
        CG algorithms benefit from sub-optimal, well-centered interior solutions from the master problem \cite{Gondzio1996_PDCGM}, which are typically obtained by simply relaxing an IPM solver's optimality tolerance.
        These provide the double benefit of stabilizing the CG procedure, thus reducing the number of CG iterations, and speeding-up the resolution of the master problem by stopping the IPM early.
        Importantly, this approach requires \emph{feasible}, but sub-optimal, dual solutions from the master problem.
        While in classical primal-dual IPMs, feasibility is generally reached earlier than optimality, in the homogeneous algorithm, infeasibilities and complementarity are reduced at the same rate \cite{Andersen2000}.
        As a consequence, for IPM solvers that implement the homogeneous algorithm, such as Mosek, ECOS and Tulip, relaxing optimality tolerances yields no computational gain.
        Nevertheless, let us formally restate that our present goal is \emph{not} to implement a state-of-the-art column-generation solver, but to quantify the benefits of specialized linear algebra in that context; in particular, specialized linear algebra would equally benefit classical primal-dual IPMs, since the approach of \cite{Gondzio1996_PDCGM} does not affect the master problem's structure.
        Therefore, we only implement a vanilla CG procedure, which is described in Appendix \ref{sec:colgen}.
        In particular, we do not make use of any acceleration technique beyond the use of partial pricing.
        
        Table \ref{tab:der:stats} and Table \ref{tab:tssp:stats} display some statistics for DER and TSSP instances, respectively.
        For each instance, we report: the number of sub-problems $R$, the number of CG iterations (Iter), total time spent solving the master problem (Master) and pricing sub-problems (Pricing) during the CG procedure and, for the final $(MP)$: the number of linking constraints ($m_{0}$), the number of variables ($N$), and the proportion of non-zero coefficients in the linking constraints ($\%$nz).
        From the two tables, we see that \texttt{DER}, \texttt{4node} and \texttt{4node-base} instances display relatively dense linking rows, with $35$ to $90\%$ coefficients being non-zeros, and a modest number of linking constraints.
        Other instances are either sparser, e.g., the \texttt{env} and \texttt{env-diss} instances whose linking rows are only $13\%$ dense, or have few linking constraints, e..g, \texttt{phone}.
        Therefore, we expect that our specialized implementation will yield larger gains for the former instances.
        
        \revision{blue}{
\begin{table}
    \centering
    \caption{Column-generation statistics - DER instances}
    \label{tab:der:stats}
    \begin{tabular}{lrrrrcrrr}
\toprule
&& \multicolumn{3}{c}{CG statistics} && \multicolumn{3}{c}{MP statistics}\\
\cmidrule(rl){3-5} \cmidrule(rl){7-9}
Instance & $R$ & Iter & Master(s) & Pricing(s) && $m_{0}$  & $N$ &  $\%$nz\\
\midrule
DER-24 &  1024 &  43 &   4.5 &  16.6 &&  24 &   6493 & 89.7 \\
 &  2048 &  40 &   9.7 &  40.8 &&  24 &  12152 & 89.0 \\
 &  4096 &  41 &  24.0 &  86.5 &&  24 &  24559 & 89.4 \\
 &  8192 &  40 &  74.9 & 155.9 &&  24 &  48668 & 89.7 \\
 & 16384 &  42 & 195.8 & 419.9 &&  24 &  95845 & 90.1 \\
 & 32768 &  40 & 585.7 & 826.3 &&  24 & 192039 & 89.6 \\
DER-48 &  1024 &  49 &  10.8 &  25.7 &&  48 &   7440 & 87.0 \\
 &  2048 &  49 &  24.6 &  50.7 &&  48 &  14736 & 88.0 \\
 &  4096 &  49 &  60.8 & 103.2 &&  48 &  29328 & 88.3 \\
 &  8192 &  50 & 148.1 & 212.3 &&  48 &  59536 & 88.5 \\
 & 16384 &  48 & 355.4 & 418.3 &&  48 & 114832 & 88.5 \\
 & 32768 &  47 & 853.8 & 870.2 &&  48 & 225424 & 88.4 \\
DER-96 &  1024 &  64 &  49.0 &  67.9 &&  96 &   9504 & 86.7 \\
 &  2048 &  56 &  90.8 & 117.0 &&  96 &  16672 & 87.8 \\
 &  4096 &  53 & 191.7 & 220.1 &&  96 &  31520 & 88.2 \\
 &  8192 &  60 & 603.4 & 529.9 &&  96 &  69920 & 88.5 \\
 & 16384 &  57 & 1248.7 & 993.0 &&  96 & 133408 & 89.0 \\
 & 32768 &  54 & 3657.2 & 2163.7 &&  96 & 254240 & 88.7 \\
\bottomrule
    \end{tabular}
\end{table}
}
        \revision{blue}{
\begin{table}
    \centering
    \caption{Column-generation statistics - TSSP instances}
    \label{tab:tssp:stats}
    \begin{tabular}{lrrrrcrrr}
\toprule
&& \multicolumn{3}{c}{CG statistics} && \multicolumn{3}{c}{MP statistics}\\
\cmidrule(rl){3-5} \cmidrule(rl){7-9}
Instance & $R$ & Iter & Master(s) & Pricing(s) && $m_{0}$  & $N$ & $\%$nz\\
\midrule
     4node &  1024 &  24 &   3.4 &   2.9 &&  60 &   5997 & 41.5 \\
      &  2048 &  24 &   7.9 &   7.2 &&  60 &  11614 & 38.8 \\
      &  4096 &  22 &  14.4 &  14.0 &&  60 &  22034 & 38.0 \\
      &  8192 &  23 &  43.8 &  28.1 &&  60 &  44691 & 37.3 \\
      & 16384 &  23 & 114.8 &  58.0 &&  60 &  87569 & 37.9 \\
      & 32768 &  21 & 248.7 &  95.8 &&  60 & 158895 & 36.5 \\
4node-base &  1024 &  26 &   4.5 &   2.8 &&  60 &   6197 & 59.4 \\
 &  2048 &  27 &  11.9 &   5.5 &&  60 &  12968 & 60.2 \\
 &  4096 &  25 &  35.8 &  10.5 &&  60 &  24153 & 60.3 \\
 &  8192 &  22 &  46.6 &  17.4 &&  60 &  43399 & 59.4 \\
 & 16384 &  25 & 143.8 &  45.0 &&  60 &  95792 & 60.4 \\
 & 32768 &  23 & 321.6 &  79.8 &&  60 & 179472 & 60.2 \\
    assets & 37500 &   6 &   2.1 &   6.2 && 13 &  77928 & 38.5 \\
       env &  1200 &   6 &   0.1 &   0.5 &&  85 &   2860 & 12.5 \\
        &  1875 &   6 &   0.1 &   0.7 &&  85 &   4283 & 12.9 \\
        &  3780 &   6 &   0.2 &   1.5 &&  85 &   8357 & 13.3 \\
        &  5292 &   6 &   0.2 &   2.1 &&  85 &  11541 & 13.5 \\
        &  8232 &   6 &   0.4 &   3.3 &&  85 &  17664 & 13.6 \\
        & 32928 &   6 &   2.5 &  13.2 &&  85 &  69783 & 13.8 \\
  env-diss &  1200 &  13 &   0.2 &   0.7 &&  85 &   4439 & 12.3 \\
   &  1875 &  15 &   0.4 &   1.3 &&  85 &   7435 & 12.7 \\
   &  3780 &  15 &   1.0 &   2.6 &&  85 &  15168 & 12.9 \\
   &  5292 &  15 &   1.5 &   3.6 &&  85 &  20745 & 13.0 \\
   &  8232 &  15 &   2.5 &   5.7 &&  85 &  31752 & 13.0 \\
   & 32928 &  14 &  14.8 &  21.8 &&  85 & 123892 & 13.3 \\
     phone & 32768 &   5 &   1.4 &   7.6 &&   9 &  65553 & 83.3 \\
   stormG2 &  1000 &  21 &   7.9 &  10.9 && 306 &   6075 & 23.9 \\
\bottomrule
    \end{tabular}
\end{table}
}
        
        }
    
\revision{red}{
    \subsubsection{Specialized linear algebra}
    \label{sec:res:colgen:linalg}
}
    
    We now describe a specialized Cholesky factorization that exploits the block structure of the master problem.
	First, the constraint matrix of $(MP)$ is unit block-angular, i.e., it has the form
	\begin{align}
		\label{eq:struct_A_unitblockangular}
		A & = 
		\begin{bmatrix}
			e^{T} 	& &&0\\
								& \ddots  & & \vdots\\
								&		 & e^{T} & 0\\
			A_{1} 		& \cdots & A_{R} & A_{0}
		\end{bmatrix},
	\end{align}
	where
	\begin{align}
		A_{r} &=
		\begin{pmatrix}
			|					& 		& |\\
			a_{r, 1}	& \dots & a_{r, n_{r}}\\
			| 					&		& |
		\end{pmatrix}
		\in \mathbb{R}^{m_{0} \times n_{r}}
		.
	\end{align}

	Let us recall that the normal equations system writes
	\begin{align}
	    \label{eq:normaleq}
		\revision{red}{\left(A(\Theta^{-1}+\prho I)^{-1} A^{T} + \drho I \right) \delta_{y}} = \xi,
	\end{align}
	where \revision{red}{$\delta_{y} \in \mathbb{R}^{M}$}, and $\Theta \in \mathbb{R}^{N \times N}$ is a diagonal matrix with positive diagonal.
	\revision{red}{Let $S$ denote the left-hand matrix of \eqref{eq:normaleq}, and define
	\begin{align}
		\label{eq:struct_D_blockdiagonal}
		\tilde{\Theta} = (\Theta^{-1}+\prho I)^{-1} = 
		\begin{pmatrix}
			\tilde{\Theta}_{1}	& 		 & \\
						& \ddots	 & \\
						& 		 & \tilde{\Theta}_{R}\\
						&&& \tilde{\Theta}_{0}
		\end{pmatrix}
		,
	\end{align}
	and $\tilde{\theta}_{r} = \tilde{\Theta}_{r}e \in \R^{n_{r}}$, for $r=0, ..., R$.}
	Consequently, the normal equations system has the form
	\revision{red}{
	\begin{align}
		\label{eq:normalEq_blocks}
		\begin{bmatrix}
			d_{1}   && && && (A_{1}\tilde{\theta}_{1})^{T}\\
									&& \ddots 	&&							&& \vdots\\
									&&			&& 	d_{R} 			&& (A_{R}\tilde{\theta}_{R})^{T}\\
			A_{1}\tilde{\theta}_{1} 	&& \cdots 	&& A_{R} \tilde{\theta}_{R} 	&& \Phi
		\end{bmatrix}
		\begin{bmatrix}
			(\delta_{y})_{1}\\
			\vdots\\
			(\delta_{y})_{R}\\
			(\delta_{y})_{0}
		\end{bmatrix}
		&=
		\begin{bmatrix}
			\xi_{1}\\
			\vdots\\
			\xi_{R}\\
			\xi_{0}
		\end{bmatrix},
	\end{align}}
	where \revision{red}{
	\begin{align}
	    d_{r} &= e^{T}\tilde{\theta}_{r} + \drho, \ \ \ r=1, ..., R,\\
	    \Phi &= \sum_{r=0}^{R} A_{r} \tilde{\Theta}_{r} A_{r}^{T} + \drho I.
	\end{align}
	}

	Then, define
	\begin{align}
		l_{r} & = \frac{1}{\revision{red}{d_{r}}} A_{r} \tilde{\theta}_{r} \in \mathbb{R}^{m_{0}}, \ \ \ r=1, ..., R,\\
		\label{eq:schur_complement}
		C &= \Phi - \sum_{r=1}^{R} \frac{1}{\revision{red}{d_{r}}} (A_{r} \tilde{\theta}_{r}) (A_{r} \tilde{\theta}_{r})^{T} \in \mathbb{R}^{m_{0} \times m_{0}}.
	\end{align}
	Given that both \revision{red}{$S$} and its upper-left block are positive definite, so is the Schur complement $C$.
	Therefore, its Cholesky factorization exists, which we denote $C = L_{C} D_{C} L_{C}^{T}$.
	It then follows that a Cholesky factorization of \revision{red}{$S$} is given by 
	\begin{align}
		\revision{red}{S} & = 
		\underbrace{
		\begin{bmatrix}
			1 & \\
				& \ddots \\
				&	& 1\\
			l_{1} & \cdots & l_{R} & L_{C}
		\end{bmatrix}
		}_{L}
		\times
		\underbrace{
		\begin{bmatrix}
			\revision{red}{d_{1}} & & \\
				& \ddots 	&	\\
				&	& \revision{red}{d_{R}} \\
			 &  & & D_{C}
		\end{bmatrix}
		}_{D}
		\times
		\underbrace{
		\begin{bmatrix}
			1 & \\
				& \ddots \\
				&	& 1\\
			l_{1} & \cdots & l_{R} & L_{C}
		\end{bmatrix}^{T}
		}_{L^{T}}
		.
	\end{align}

	Finally, once the Cholesky factors $L$ and $D$ are computed, the normal equations \eqref{eq:normalEq_blocks} are solved as follows:
	\begin{align}
		\label{eq:normalEq_globalSolve}
		\revision{red}{(\delta_{y})_{0}} &= (L_{C} D_{C} L_{C}^{T})^{-1} \left( \xi_{0} - \sum_{r=1}^{R} \xi_{r} l_{r} \right), \\
		\label{eq:normalEq_localSolve}
		\revision{red}{(\delta_{y})_{r}} &= \dfrac{1}{\revision{red}{d_{r}}} \xi_{r} - l_{r}^{T} \revision{red}{(\delta_{y})_{0}}, \ \ \ r=1, ..., R.
	\end{align}

	Exploiting the structure of $A$ yields several computational advantages.
	First, the factors $L$ and $D$ can be computed directly from $A$ and $\Theta$, i.e., the matrix \revision{red}{$S$} does not need to be explicitly formed nor stored, thus saving both time and memory.
	Second, the sparsity structure of $L$ is known beforehand.
	Specifically, the lower blocks $l_{1}, \dots, l_{R}$ are all dense column vectors, and the Schur complement $C$ is a dense $m_{0}\times m_{0}$ matrix.
	Therefore, one does not need a preprocessing phase wherein a sparsity-preserving ordering is computed, thus saving time and making memory allocation fully known in advance.
	Third, since most heavy operations are performed on dense matrices, efficient cache-exploiting kernels for dense linear algebra can be used, further speeding-up the computations.
	Finally, note that most operations such as forming the Cholesky factors and performing the backward substitutions, are amenable to parallelization.
    
\revision{blue}{
    \subsubsection{Experimental setup}
    \label{sec:res:colgen:setup}
        
        We implement the specialized routines described above in Julia.\footnote{\url{https://github.com/mtanneau/UnitBlockAngular.jl}}
        Specifically, we define a \texttt{UnitBlockAngularMatrix} type, together with specialized matrix-vector product methods, and a \texttt{UnitBlockAngularFactor} type for computing factorizations and solving linear systems.
        Dense linear algebra operations are performed by BLAS/LAPACK routines directly, and the entire implementation is less than 250 lines of code.
        
        This specialized implementation is passed to the solver by setting the \texttt{MatrixOptions} and \texttt{KKTOptions} parameters accordingly, as illustrated in Figure \ref{code:KKT:unitblockangular}.
        A \texttt{Model} object is first created at line 4, and the problem data is imported at line 5.
        At line 11, we set the \texttt{MatrixOptions} parameter to specify that the constraint matrix is of the \texttt{UnitBlockAngularMatrix} type with $m_{0}=24$ linking constraints, $n_{0}=72$ linking variables, $n=6421$ non-linking variables, and $R=1024$ unit blocks.
        Then, at line 16, we select the \texttt{UnitBlockAngularFactor} type as a linear solver.
        Finally, the correct matrix and linear solver are instantiated within the \texttt{optimize!} call at line 20.
        Importantly, let us emphasize that no modification was made to Tulip's source code: the correct methods are automatically selected by Julia's multiple dispatch feature, with no performance loss for calling an external function.
        
        Experiments are carried out on an Intel Xeon E5-2637@3.50GHz CPU, 128GB RAM machine running Linux; scripts and data for running these experiments are available online.\footnote{Code for generating DER instances is available at \url{https://github.com/mtanneau/DER_experiments} and for TSPP instances at \url{https://github.com/mtanneau/TSSP}}
        We compare the following IPM solvers: CPLEX 12.10 \cite{CPLEX}, Gurobi 9.0 \cite{Gurobi}, Mosek 9.2.5 \cite{Mosek}, Tulip 0.5.0 with generic linear algebra, and Tulip 0.5.0 with specialized linear algebra; the latter is denoted Tulip*.
        We run each solver on a single thread, and no crossover.
        Presolve may alter the structure of $A$ in several ways by, e.g., reducing the number of linking constraints, eliminating variables --possibly some entire blocks-- or modifying the unit blocks during scaling.
        Therefore, since we are interested in comparing the per-iteration cost among solvers, we also deactivate presolve.
        Finally, none of the selected IPM solvers have any warm-start capability, i.e., in a CG algorithm, master problems would effectively be solved from scratch at each CG iteration.
        Thus, solving master problems independently of one another, as is done here, does not invalidate our analysis.
}

\begin{figure}
    \centering
\begin{lstlisting}[language=Python]
import Tulip
using UnitBlockAngular

model = Tulip.Model{Float64}()
Tulip.load_problem!(model, "DER_24_1024_43.mps")  # read file

# Deactivate presolve
model.params.Presolve = 0

# Select matrix options
model.params.MatrixOptions = Tulip.TLA.MatrixOptions(
    UnitBlockAngularMatrix,
    m0=24, n0=72, n=6421, R=1024
)
# Select custom linear solver
model.params.KKTOptions = Tulip.KKT.SolverOptions(
    UnitBlockAngularFactor
)

Tulip.optimize!(model)  # solve the problem
\end{lstlisting}
\caption{Sample Julia code illustrating the use of a custom \texttt{UnitBlockAngularMatrix} type and specialized factorization.}
\label{code:KKT:unitblockangular}
\end{figure}

\revision{blue}{
    \subsubsection{Results}
    \label{src:res:colgen:res}
        
        Results are reported in Table \ref{tab:res:rmp:short}; for conciseness, only the final master problem of each CG instance is included here.
        Results for the entire collection can be found in Table \ref{tab:res:rmp:full}, Appendix \ref{app:res}.
        For each instance and solver, we report total CPU time (T), in seconds, and the number of IPM iterations (Iter).
        In Table \ref{tab:res:rmp:full}, the number of CG iterations (at which the instance was obtained) is also displayed.
        
        \revision{blue}{
\begin{table}
    \centering
    \caption{Performance comparison of IPM solvers on structured instances}
    \label{tab:res:rmp:short}
    \begin{tabular}{lrrrrrrrrrrr}
\toprule
 & & \multicolumn{2}{c}{CPLEX} & \multicolumn{2}{c}{Gurobi} & \multicolumn{2}{c}{Mosek} & \multicolumn{2}{c}{Tulip} & \multicolumn{2}{c}{Tulip*}\\
\cmidrule(rl){3-4} \cmidrule(rl){5-6} \cmidrule(rl){7-8} \cmidrule(rl){9-10} \cmidrule(rl){11-12}
Problem          &   $R$ & T(s) & Iter & T(s) & Iter & T(s) & Iter & T(s) & Iter & T(s) & Iter\\
\midrule
DER-24           &  1024 &    0.2 &  33 &    0.2 &  27 & \textbf{   0.2} &  21 &    1.1 &  33 &    0.5 &  33\\
           &  2048 &    0.4 &  48 & \textbf{   0.4} &  36 &    0.4 &  27 &    2.5 &  47 &    0.6 &  47\\
           &  4096 &    1.0 &  40 &    1.0 &  32 & \textbf{   0.8} &  26 &    4.7 &  38 &    1.0 &  38\\
           &  8192 &    4.3 &  79 &    2.7 &  46 & \textbf{   2.4} &  38 &   19.0 &  67 &    2.6 &  68\\
           & 16384 &   10.8 &  93 &    5.3 &  48 & \textbf{   5.3} &  42 &   49.8 &  86 &    5.3 &  83\\
           & 32768 &   33.9 & 148 &   19.4 &  85 &   12.3 &  43 &  103.6 &  91 & \textbf{  11.4} &  86\\
DER-48           &  1024 &    0.5 &  37 &    0.4 &  19 & \textbf{   0.3} &  21 &    1.8 &  28 &    0.4 &  28\\
           &  2048 &    1.6 &  40 &    1.0 &  21 & \textbf{   0.8} &  25 &    5.7 &  37 &    0.9 &  37\\
          &  4096 &    4.1 &  44 &    2.0 &  25 &    2.0 &  27 &   14.1 &  39 & \textbf{   1.7} &  39\\
          &  8192 &    9.7 &  51 &    4.2 &  20 &    4.5 &  24 &   37.0 &  46 & \textbf{   3.4} &  47\\
           & 16384 &   22.3 &  64 &    9.9 &  29 &    9.3 &  28 &   89.7 &  60 & \textbf{   7.4} &  57\\
           & 32768 &   57.1 &  85 &   21.6 &  32 &   21.1 &  33 &  178.8 &  59 & \textbf{  14.2} &  54\\
DER-96           &  1024 &    3.3 &  38 &    1.2 &  19 &    0.9 &  22 &    6.6 &  31 & \textbf{   0.9} &  31\\
           &  2048 &    7.9 &  45 &    2.4 &  20 &    1.7 &  21 &   18.2 &  38 & \textbf{   1.7} &  37\\
           &  4096 &   16.3 &  51 &    5.5 &  24 &    5.2 &  28 &   42.6 &  40 & \textbf{   3.2} &  40\\
           &  8192 &   51.7 &  75 &   15.5 &  29 &   11.1 &  31 &  137.6 &  60 & \textbf{   8.8} &  57\\
           & 16384 &  107.8 &  86 &   31.9 &  31 &   24.4 &  39 &  260.0 &  55 & \textbf{  17.3} &  59\\
           & 32768 &  291.9 & 119 &  102.9 &  54 & \textbf{  55.5} &  47 &  753.7 &  89 &   65.4 &  86\\
4node            &  1024 &   15.7 &  21 &    0.4 &  43 & \textbf{   0.2} &  25 &    1.3 &  30 &    0.5 &  32\\
            &  2048 &    0.7 &  38 &    0.6 &  27 & \textbf{   0.6} &  25 &    2.1 &  36 &    0.9 &  36\\
            &  4096 &    1.1 &  27 &    1.7 &  37 & \textbf{   0.7} &  17 &    4.5 &  28 &    1.2 &  28\\
            &  8192 &    2.7 &  30 &    2.3 &  29 & \textbf{   1.8} &  24 &   12.3 &  35 &    2.7 &  33\\
            & 16384 &    5.8 &  29 &   10.7 &  53 & \textbf{   4.0} &  22 &   26.4 &  33 &    4.6 &  33\\
            & 32768 &   17.0 &  57 &   18.7 &  55 &   14.6 &  41 &   74.8 &  56 & \textbf{  14.2} &  59\\
4node-base       &  1024 &   17.0 &  17 &    1.0 &  60 & \textbf{   0.3} &  27 &    1.4 &  28 &    0.6 &  27\\
       &  2048 &    1.0 &  35 &    2.6 &  72 & \textbf{   0.8} &  33 &    3.7 &  32 &    0.9 &  33\\
       &  4096 &    2.3 &  38 &    5.3 &  72 & \textbf{   1.5} &  34 &    9.1 &  34 &    1.8 &  34\\
       &  8192 &    3.8 &  29 &    3.7 &  27 & \textbf{   2.6} &  25 &   19.7 &  36 &    2.8 &  36\\
       & 16384 &   13.5 &  53 &   26.2 &  74 &    8.0 &  37 &   63.4 &  53 & \textbf{   7.0} &  47\\
       & 32768 &   20.3 &  37 &   29.0 &  43 &   14.9 &  30 &  107.7 &  48 & \textbf{  12.9} &  50\\
assets           & 37500 &    1.6 &  21 & \textbf{   0.6} &  12 &    1.1 &  20 &    2.0 &  13 &    1.0 &  13\\
env              &  1200 &    0.0 &  21 & \textbf{   0.0} &  12 &    0.1 &  16 &    0.3 &  16 &    0.3 &  16\\
              &  1875 &    0.1 &  22 & \textbf{   0.0} &  12 &    0.1 &  13 &    0.4 &  16 &    0.4 &  16\\
              &  3780 &    0.1 &  25 & \textbf{   0.1} &  12 &    0.1 &  14 &    0.7 &  17 &    0.5 &  17\\
              &  5292 &    0.2 &  27 & \textbf{   0.1} &  13 &    0.1 &  13 &    0.7 &  17 &    0.7 &  17\\
              &  8232 &    0.3 &  26 & \textbf{   0.2} &  13 &    0.3 &  14 &    1.1 &  18 &    1.2 &  18\\
              & 32928 &    1.7 &  26 & \textbf{   0.9} &  13 &    1.3 &  17 &    5.1 &  21 &    4.4 &  21\\
env-diss         &  1200 &    0.1 &  15 & \textbf{   0.0} &  15 &    0.1 &  17 &    0.4 &  23 &    0.4 &  23\\
         &  1875 &    0.1 &  17 & \textbf{   0.1} &  18 &    0.1 &  18 &    0.6 &  22 &    0.5 &  22\\
         &  3780 &    0.2 &  20 & \textbf{   0.1} &  18 &    0.2 &  18 &    1.0 &  22 &    1.0 &  22\\
         &  5292 &    0.3 &  22 & \textbf{   0.3} &  23 &    0.3 &  22 &    1.3 &  25 &    1.5 &  25\\
         &  8232 &    1.0 &  31 &    0.6 &  29 & \textbf{   0.5} &  23 &    3.2 &  35 &    2.6 &  35\\
         & 32928 &    4.8 &  28 & \textbf{   2.1} &  22 &    2.5 &  19 &   10.0 &  27 &    7.7 &  27\\
phone            & 32768 &    0.5 &  15 & \textbf{   0.4} &   8 &    0.6 &   8 &    1.9 &  10 &    0.7 &  10\\
stormG2          &  1000 &    1.6 &  37 &    0.8 &  18 & \textbf{   0.5} &  22 &    4.0 &  29 &    1.7 &  28\\
\bottomrule
    \end{tabular}\\
    Results obtained without presolve.
\end{table}
}
        
        We begin by comparing Tulip with and without specialized linear algebra.
        First, the number of IPM iterations is almost identical between the two, with differences never exceeding 6 IPM iterations.
        The differences are caused by small numerical discrepancies between the linear algebra implementations, which remain negligible until close to the optimum.
        Second, using specialized linear algebra results in a significant speedup, especially on larger and denser instances.
        Indeed, on large \texttt{DER} and \texttt{4node} instances, we typically observe a tenfold speedup.
        For smaller and sparser instances, e.g., the \texttt{env} instances, or with very few linking constraints such as \texttt{phone}, using specialized linear algebra still brings a moderate performance improvement.
        
        Next, we compare Tulip with specialized linear algebra, Tulip*, against state-of-the-art commercial solvers.
        Given CPLEX's poorer relative performance on this test set, in the following we mainly discuss the results of Tulip* in comparison with Mosek and Gurobi.
        First, our specialized implementation is able to outperform commercial codes on the larger and denser instances, while remaining within a reasonable factor on smaller and sparse instances.
        The largest performance improvement is observed on the \texttt{DER-48} instance with $R=32,768$, for which Tulip* achieves a $30 \%$ speedup over the fastest commercial alternative.
        This demonstrates that, when exploiting structure, open-source solvers can compete with state-of-the-art commercial codes.
        Second, Tulip's iteration count is typically $50$ to $100 \%$ larger than that of Mosek and Gurobi.
        When comparing average per-iteration times on the denser instances, we observe that Tulip is generally $1.5$ to $3$ times faster than Gurobi and Mosek.
}
        Recall that the cost of an individual IPM iteration depends not only on problem size and the efficiency of the underlying linear algebra, but also on algorithmic features such as the number of corrections, which we cannot \revision{blue}{measure} directly.
    \revision{blue}{Nevertheless, the performance difference is significant enough to suggest that algorithmic improvements aimed at reducing the number of IPM iterations would substantially improve Tulip's performance.}

{
\subsection{Solving problems in extended precision}
\label{sec:res:precision}

    Almost all optimization solvers perform computations in double precision (64 bits) floating-point arithmetic, denoted by \texttt{double} and \texttt{Float64} in C and Julia, respectively.
    Julia's parametric type system and multiple dispatch allow to write generic code: in the present case, this results in Tulip's code can be used with \emph{arbitrary} arithmetic.
    We now illustrate this functionality for solving problems in higher precision.
    
    The ability to use extended precision is useful is various contexts.
    First, while typical numerical tolerances for most LP solvers range from $10^{-6}$ to $10^{-8}$, one may \emph{require} levels of precision that exceed what double-precision arithmetic can achieve.
    For instance, in \cite{Ma2015}, the authors consider problems where variations of order $10^{-6}$ to $10^{-10}$ are meaningful.
    One remedy to this issue is to use, e.g., quadruple-precision arithmetic.
    Second, even with ``standard" tolerances, solvers may encounter numerical issues for badly scaled problems, sometimes resulting in the optimization being aborted.
    These issues may be alleviated by using higher precision, thereby allowing to solve a given challenging instance, albeit at a performance cost.
    Finally, in the course of developing a new optimization software or algorithmic technique, identifying whether inconsistencies are due to numerical issues, mathematical errors, or software bugs, can be a daunting and time-consuming task.
    In that context, the ability to easily switch between different arithmetics enables one to factor out rounding errors and related issues, thereby identifying --or ruling out-- other sources of errors.
    
    Let us note that a handful of simplex-based solvers have the capability to compute extended-precision or exact solutions to LP problems, either by performing computations in exact arithmetic, solving a sequence of LPs with increasing precision, or using iterative refinement techniques; the reader is referred to \cite{Gleixner2016} for an overview of such approaches and available software.
    We are not aware of any existing interior-point solver with this capability.
    As pointed out in \cite{Gleixner2016}, performing all computations in the prescribed arithmetic, as is the case in Tulip, is intractable for large problems.
    Consequently, Tulip should not be viewed as a competitive tool for solving LPs in extended precision.
    Rather, the main advantage of our implementation is its simplicity and flexibility: it required no modification of the source code, runs the same algorithm regardless of the arithmetic, and its use is straightforward.
    Indeed, as Figure \ref{fig:arithmetic} illustrates, besides loading the appropriate packages, the user only needs to specify the arithmetic when creating a model; the rest of the code is identical.
    Therefore, using Tulip with higher-precision arithmetic is best envisioned as a prototyping tool, or to occasionally solve a numerically challenging problem.
    
\begin{figure}
    \centering
    \begin{subfigure}{\textwidth}
\begin{lstlisting}[language=Python]
import Tulip

model = Tulip.Model{Float64}()  # Float64 arithmetic
Tulip.load_problem!(tlp, "neos2.mps")  # read file

Tulip.optimize!(model)  # solve the problem
\end{lstlisting}
    \caption{Using \texttt{Float64} arithmetic.}
    \label{fig:arithmetic:F64}
    \end{subfigure}
    
    \begin{subfigure}{\textwidth}
    \centering
\begin{lstlisting}[language=Python]
import Tulip
using DoubleFloats

model = Tulip.Model{Double64}()  # Double64 arithmetic
Tulip.load_problem!(tlp, "neos2.mps")  # read file

Tulip.optimize!(model)  # solve the problem
\end{lstlisting}
    \caption{Using \texttt{Double64} arithmetic.}
    \label{fig:arithmetic:D64}
    \end{subfigure}
    \caption{Sample Julia code illustrating the use of different arithmetics.}
    \label{fig:arithmetic}
\end{figure}

    As an example of this use case, we consider the $6$ instances from Section \ref{sec:res:plato} that required more than $100$ IPM iterations; this generally indicates numerical issues.
    Each instance is solved with Tulip in quadruple-precision arithmetic.
    We use the \texttt{Double64} type from the \texttt{DoubleFloats} Julia package, which implements the so-called ``double-double" arithmetic, wherein a pair of double-precision numbers is used to approximate one quadruple-precision number.
    This implementation allows to exploit fast, hardware-implemented, double-precision arithmetic, while achieving similar precision as 128 bits floating point arithmetic.
    Experiments were carried on the same cluster of machines as in Section \ref{sec:res:plato}.
    Besides the different arithmetic, we increase the time limit to $40,000$s and set tolerances to $10^{-8}$, that is, the problems are solved up to usual double-precision tolerances.
    All other settings are left identical to those of Section \ref{sec:res:plato}.

    Results are displayed in Table \ref{tab:res:plato:Double64}.
    For each instance and arithmetic, we report the total solution time (CPU) in seconds, the number of IPM iterations (Iter), and the solver's result status (Status).
    We first note that, when using \texttt{Double64} arithmetic, all instances are solved to optimality.
    This validates the earlier finding that instances \texttt{ns1688926} and \texttt{watson\_2} did encounter numerical issues.
    Second, we observe a drastic reduction in the number of IPM iterations from \texttt{Float64} to \texttt{Double64}, with decreases in iteration counts ranging from $40\%$ to over $90\%$ in the case of \texttt{neos2} and \texttt{ns1688926}.
    Third, while the per-iteration cost of \texttt{Double64} is typically 8x larger than that of \texttt{Float64}, overall computing times do not increase as much due to the reduction in IPM iterations.
    In fact, in the extreme cases of \texttt{ns1688926}, solving the problem in \texttt{Double64} is significantly faster than solving it in \texttt{Float64}.
    Finally, the results of Table \ref{tab:res:plato:Double64} suggest that Tulip would most benefit from greater numerical stability on instances such as \texttt{neos2}, \texttt{ns1688926}, \texttt{stat96v1} and \texttt{watson\_2}.
    This may include, for instance, the use of iterative refinement when solving Newton systems.
    On the other hand, similar iterations counts for both arithmetics, would have suggested algorithmic issues, e.g., short steps being taken due to the iterates being far from the central path.

    \begin{table}
        \centering
        \caption{Problematic instances from the Mittelmann benchmark}
        \label{tab:res:plato:Double64}
        \begin{tabular}{lcrrrcrrr}
        \toprule
        && \multicolumn{3}{c}{\texttt{Float64}} && \multicolumn{3}{c}{\texttt{Double64}}\\
        \cmidrule(rl){3-5} \cmidrule(rl){7-9}
        Instance && CPU (s) & Iter & Status && CPU(s) & Iter & Status\\
        \midrule
            \texttt{neos2}      &&  462.1 & 460 & Optimal && 265.1 & 37 & Optimal\\
            \texttt{ns1688926}  &&  1007.7 & 500 & Iterations && 142.8 & 18 & Optimal\\
            \texttt{s250r10}    &&  257.2 & 169 & Optimal && 1385.0 & 93 & Optimal\\
            \texttt{shs1023}    &&  371.6 & 266 & Optimal && 968.8 & 105 & Optimal\\
            \texttt{stat96v1}   &&  41.3 & 275 & Optimal && 30.4 & 42 & Optimal\\
            \texttt{watson\_2}    &&  295.7 & 500 & Iterations && 243.4 & 67 & Optimal\\
        \bottomrule
        \end{tabular}
    \end{table}
}

\section{Conclusion}
\label{sec:conclusion}

In this paper, we have described a \revision{red}{regularized} homogeneous interior-point algorithm and its implementation in Tulip, an open-source linear optimization solver.
\revision{blue}{Our solver is written in Julia, and leverages some of the language's features to propose a flexible and easily-customized implementation.}
Most notably, \revision{blue}{Tulip's} algorithmic framework is fully disentangled from linear algebra implementations \revision{blue}{and the choice of arithmetic}.

The performance of the code has been evaluated on \revision{blue}{generic instances from H. Mittelmann's benchmark} testset, on \revision{blue}{two sets} of structured instances for which we developed specialized linear algebra routines, \revision{blue}{and on numerically problematic instances using higher-precision arithmetic.
The computational evaluation has shown three main results.
First, when solving generic LP instances, Tulip is competitive with open-source IPM solvers that have a Julia interface.
Second, when solving structured problems, the use of custom linear algebra routines yields a tenfold speedup over generic ones, thereby outperforming state-of-the-art commercial IPM solvers on larger and denser instances.
}
These results demonstrate the benefits of being able to seamlessly integrate specialized linear algebra within an interior-point algorithm.
\revision{blue}{Third, in a development context, Tulip can be conveniently used in conjunction with higher-precision arithmetic, so as to alleviate numerical issues.}

\revision{blue}{%
Finally, future developments will consider the use of iterative methods for solving linear systems, the development of more general structured linear algebra routines and their multi-threaded implementation, and more efficient algorithmic techniques for solving problems in extended precision.
Because of the way in which Tulip has been designed, all those developments do not require any significant rework of the code structure.}

\section*{Acknowledgements}
    We thank Dominique Orban for helpful discussions on the regularization scheme and its implementation.
    We are also indebted to three anonymous referees for their careful reading and constructive suggestions that helped us improving the quality and readability of the paper.

\bibliographystyle{spmpsci}      
\bibliography{refs}

\newpage
\appendix

\section{Dantzig-Wolfe decomposition and column generation}
\label{sec:colgen}

In this section, we present the Dantzig-Wolfe decomposition principle \cite{Dantzig1960} and the basic column-generation framework.
We refer to \cite{Desaulniers2006column} for a thorough overview of column generation, and the relation between Dantzig-Wolfe decomposition and Lagrangian decomposition.

\subsection{Dantzig-Wolfe decomposition}
\label{sec:colgen:subsec:dw}
	\revision{blue}{Consider} the problem
	\begin{align*}
		(P) \ \ \ \min_{x} \ \ \  & \sum_{r=0}^{R} c_{r}^{T} x_{r}\\
		s.t. \ \ \ & \sum_{r=0}^{R} A_{r} x_{r} = b_{0},\\
		& \revision{blue}{x_{0} \geq 0,}\\
		& x_{r} \in \mathcal{X}_{r}, \ \ \ \revision{blue}{r=1, ..., R},
	\end{align*}
	\revision{blue}{where}, for each $r\revision{blue}{=1, ..., R}$, $\mathcal{X}_{r}$ is defined by a finite number of linear inequalities, plus integrality restrictions on some of the coordinates of $x_{r}$.
	Therefore, \revision{blue}{the convex hull of $\mathcal{X}_{r}$, denoted by $conv(\mathcal{X}_{r})$, is a polyhedron whose set of extreme points (resp. extreme rays) is denoted by $\Omega_{r}$ (resp. $\Gamma_{r}$).
	Any element of $conv(\mathcal{X}_{r})$ can thus be written as a convex combination of extreme points $\{ \omega \}_{\omega \in \Omega_{r}}$, plus a non-negative combination of extreme rays $\{ \rho \}_{\rho \in \Gamma_{r}}$ i.e.,}
	\begin{align}
		conv(\mathcal{X}_{r}) = 
		\left\{
			\sum_{\omega \in \Omega_{r}} \lambda_{\omega} \omega
			\revision{blue}{+ \sum_{\rho \in \Gamma_{r}} \lambda_{\rho} \rho}
			\ \middle| \ 
			\lambda \geq 0,
			\sum_{\omega}\lambda_{\omega} = 1
		\right\}.
	\end{align}
	
	The Dantzig-Wolfe decomposition principle \cite{Dantzig1960} then consists in \revision{blue}{substituting $x_{r}$ with such a combination of extreme points and extreme rays.}
	This change of variable yields the \revision{blue}{so-called} \emph{Master Problem}
	\begin{align}
		\label{eq:MP_objective}
		(MP) \ \ \ \min_{x, \lambda} \ \ \ 
		& \revision{blue}{c_{0}^{T}x_{0} +} \sum_{r=1}^{R} \sum_{\omega \in \Omega_{r}} c_{r, \omega} \lambda_{r, \omega} 
		    + \sum_{r=1}^{R} \sum_{\rho \in \Gamma_{r}} c_{r, \rho} \lambda_{r, \rho} \\	
		\label{eq:MP_convexity}	
		s.t. \ \ \ 
		& \sum_{\omega \in \Omega_{r}} \lambda_{r, \omega} =1, \ \ \ \revision{blue}{r=1, ..., R}\\
		\label{eq:MP_linking}
		& \revision{blue}{A_{0}x_{0}} + \sum_{r=1}^{R} \sum_{\omega \in \Omega_{r}} a_{r, \omega} \lambda_{r, \omega} + \sum_{r=1}^{R} \sum_{\rho \in \Gamma_{r}} a_{r, \rho} \lambda_{r, \rho} = b_{0},\\	
		\label{eq:MP_positive}
		& \revision{blue}{x_{0}}, \lambda \geq 0,\\
		\label{eq:MP:integer}
		& \revision{blue}{\sum_{\omega \in \Omega_{r}} \lambda_{r, \omega} \omega
			+ \sum_{\rho \in \Gamma_{r}} \lambda_{r, \rho} \rho \in \mathcal{X}_{r}, \ \ \ r=1, ..., R}
	\end{align}
	where $c_{r, \omega} = c_{r}^{T} \omega$, \revision{blue}{$c_{r, \rho} = c_{r}^{T} \rho$}, and $a_{r, \omega} = A_{r} \omega$, \revision{blue}{$a_{r, \rho} = A_{r} \rho$}.
	\revision{blue}{Constraints \eqref{eq:MP_convexity} and \eqref{eq:MP_linking} are referred to as \emph{convexity} and \emph{linking} constraints, respectively.
	
	The linear relaxation of $(MP)$ is given by \eqref{eq:MP_objective}-\eqref{eq:MP_positive}; its objective value is greater or equal to that of the linear relaxation of $(P)$ \cite{Desaulniers2006column}.
	Note that if $(P)$ in a linear program, i.e., all variables are continuous, then constraints \eqref{eq:MP:integer} are redundant, and \eqref{eq:MP_objective}-\eqref{eq:MP_positive} is equivalent to $(P)$.
	In the mixed-integer case, problem \eqref{eq:MP_objective}-\eqref{eq:MP_positive} is the root node in a branch-and-price tree.
	In this work, we focus on solving this linear relaxation.
	Thus, in what follows, we make a slight abuse of notation and use the term ``Master Problem" to refer to \eqref{eq:MP_objective}-\eqref{eq:MP_positive} instead.
	}
	
	
	

\subsection{Column generation}
\label{sec:colgen:subsec:colgen}
	
	The Master Problem has exponentially many variables.
	Therefore, it is typically solved by column generation, wherein only a small subset of the variables are considered.
	Additional variables are generated iteratively by solving an auxiliary sub-problem.
	
	Let $\bar{\Omega}_{r}$ \revision{blue}{(resp. $\bar{\Gamma}_{r}$)} be a small subset of $\Omega_{r}$ \revision{blue}{(resp. of $\Gamma_{r}$)}, and define the \emph{Restricted Master Problem} (RMP)
	\begin{align}
		\label{eq:RMP_objective}
		(RMP) \ \ \ \min_{\lambda} \ \ \ 
		& \revision{blue}{c_{0}^{T}x_{0} +} 
		\sum_{r=1}^{R} \sum_{\omega \in \bar{\Omega}_{r}} c_{r, \omega} \lambda_{r, \omega}
		\revision{blue}{+ \sum_{r=1}^{R} \sum_{\rho \in \bar{\Gamma}_{r}} c_{r, \rho} \lambda_{r, \rho}} \\	
		\label{eq:RMP_convexity}	
		s.t. \ \ \ 
		& \sum_{\omega \in \bar{\Omega}_{r}} \lambda_{r, \omega} =1, \ \ \ \revision{blue}{r=1, ..., R}\\
		\label{eq:RMP_linking}
		& \sum_{r=1}^{R} \sum_{\omega \in \bar{\Omega}_{r}} a_{r, \omega} \lambda_{r, \omega}
		\revision{blue}{+ \sum_{r=1}^{R} \sum_{\rho \in \bar{\Gamma}_{r}} a_{r, \rho} \lambda_{r, \rho}} = b_{0},\\	
		\label{eq:RMP_positive}
		& \revision{blue}{x_{0}}, \lambda \geq 0.
	\end{align}
	In all that follows, we assume that $(RMP)$ is feasible and bounded.
	Note that feasibility can be obtained by adding artificial slacks and surplus variables with sufficiently large cost, effectively implementing an $l_{1}$ penalty.
	If the RMP is unbounded, then so is the MP.
	
	\revision{blue}{Let $\sigma \in \mathbb{R}^{R}$ and $\pi \in \mathbb{R}^{m_{0}}$ denote the vector of dual variables associated to convexity constraints \eqref{eq:RMP_convexity} and linking constraints constraints \eqref{eq:RMP_linking}, respectively.
	Here, we assume that $(\sigma, \pi)$ is dual-optimal for $(RMP)$; the use of interior, sub-optimal dual solutions is explored in \cite{Gondzio1996_PDCGM}.
	Then, for given $r$, $\omega \in \Omega_{r}$ and $\rho \in \Gamma_{r}$, the reduced cost of variable $\lambda_{r, \omega}$ is
	\[
		\bar{c}_{r, \omega} = c_{r, \omega} - \pi^{T} a_{r, \omega} - \sigma_{r} = (c_{r}^{T} - \pi^{T} A_{r}) \omega - \sigma_{r},
	\]
	while the reduced cost of variable $\lambda_{r, \rho}$ is
	\[
		\bar{c}_{r, \rho}  = c_{r, \rho}   - \pi^{T} a_{r, \rho} = (c_{r}^{T} - \pi^{T} A_{r}) \rho.
	\]
	If $\bar{c}_{r, \omega} \geq 0$ for all $r$, $\omega \in \Omega_{r}$ and $\bar{c}_{r, \rho} \geq 0$ for all $r$, $\rho \in \Gamma_{r}$, then the current solution is optimal for the MP.
	Otherwise, a variable with negative reduced cost is added to the RMP.
	Finding such a variable, or proving that none exists, is called the \emph{pricing step}.
	
	Explicitly iterating through the exponentially large sets $\Omega_{r}$ and $\Gamma_{r}$ is prohibitively expensive.
	Nevertheless, the pricing step can be written as the following MILP:
	\begin{align}
		(SP_{r}) \ \ \ \min_{x_{r}} \ \ \  & (c_{r}^{T}  - \pi^{T}A)x_{r} - \sigma_{r}\\
		s.t. \ \ \ & x_{r} \in \mathcal{X}_{r},
	\end{align}
	which we refer to as the $r^{th}$ \emph{sub-problem}.
	If $SP_{r}$ is infeasible, then $\mathcal{X}_{r}$ is empty, and the original problem $P$ is infeasible. This case is ruled out in all that follows.
	Then, since the objective of $SP_{r}$ is linear, any optimal solution is either an extreme \revision{blue}{point} $\omega \in \Omega_{r}$ (bounded case), or an extreme ray $\rho \in \Gamma_{r}$ (unbounded case).
	The corresponding variable $\lambda_{r, \omega}$ or $\lambda_{r, \rho}$ is identified by retrieving an optimal \revision{blue}{point} or unbounded ray.
	Finally, note that all $R$ sub-problems $SP_{1}, \dots, SP_{R}$ can be solved independently from one another.
	Optimality in the Master Problem is attained when no variable with negative reduced cost can be identified from all $R$ sub-problems.
	
	We now describe a basic column-generation procedure, which is formally stated in Algorithm \ref{alg:colgen}.
	The algorithm starts with an initial RMP that contains a small subset of columns, some of which may be artificial to ensure feasibility.
	At the beginning of each iteration, the RMP is solved to optimality, and a dual solution $(\pi, \sigma)$ is obtained which is used to perform the pricing step.
	Each sub-problem is solved to identify a variable with most negative reduced cost.
	If a variable with negative reduced cost is found, it is added to the RMP; if not, the column-generation procedure stops.
	}
	
	\begin{algorithm}[H]
		\small
		\begin{algorithmic}[1]
			\REQUIRE Initial RMP
			
			\WHILE{stopping criterion not met}
				
				\STATE Solve RMP and obtain optimal dual variables $(\pi, \sigma)$
				
				\vspace{5pt}
				
				\STATE \COMMENT{\emph{Pricing step}}
				\FORALL{$r \in \mathcal{R}$}
					\STATE Solve $SP_{r}$ with the query point $(\pi, \sigma_{r})$; obtain $\omega^{*}$ or $\rho^{*}$
					\IF{$\bar{c}_{r,\omega^{*}} < 0$ or $\bar{c}_{r,\rho^{*}} < 0$}
						\STATE Add corresponding column to the RMP
					\ENDIF
				\ENDFOR
				
				\vspace{5pt}
				
				\STATE \COMMENT{\emph{Stopping criterion}}
				\IF{no column added to RMP}
					\STATE STOP
				\ENDIF
			\ENDWHILE
		\end{algorithmic}
		\caption{Column-generation procedure}
		\label{alg:colgen}
	\end{algorithm}
	
	For large instances with numerous subproblems, full pricing, wherein all subproblems are solved at each iteration, is often not the most efficient approach.
	Therefore, we implemented a partial pricing strategy, in which subproblems are solved in a random order until either all subproblems have been solved, or a user-specified number of columns with negative reduced cost have been generated.

\revision{blue}{
\section{Detailed results on structured LP instances}
\label{app:res}

\revision{blue}{
\begin{longtable}{lrrrrrrrrrrrr}
\caption{Structured instances: performance comparison of IPM solvers \label{tab:res:rmp:full}}\\
\toprule
 & & & \multicolumn{2}{c}{CPLEX} & \multicolumn{2}{c}{Gurobi} & \multicolumn{2}{c}{Mosek} & \multicolumn{2}{c}{Tulip} & \multicolumn{2}{c}{Tulip*}\\
\cmidrule(rl){4-5} \cmidrule(rl){6-7} \cmidrule(rl){8-9} \cmidrule(rl){10-11} \cmidrule(rl){12-13}
Instance          & $R$ & CG & T(s) & Iter & T(s) & Iter & T(s) & Iter & T(s) & Iter & T(s) & Iter\\
\midrule
\endfirsthead
\caption{(continued)}\\
\toprule
 & & & \multicolumn{2}{c}{CPLEX} & \multicolumn{2}{c}{Gurobi} & \multicolumn{2}{c}{Mosek} & \multicolumn{2}{c}{Tulip} & \multicolumn{2}{c}{Tulip*}\\
\cmidrule(rl){4-5} \cmidrule(rl){6-7} \cmidrule(rl){8-9} \cmidrule(rl){10-11} \cmidrule(rl){12-13}
Instance          & $R$ & CG & T(s) & Iter & T(s) & Iter & T(s) & Iter & T(s) & Iter & T(s) & Iter\\
\midrule
\endhead
\midrule
\endfoot
\bottomrule
\endlastfoot
DER-24           &  1024 & 10 & \textbf{   0.0} &  19 &    0.0 &  15 &    0.1 &  19 &    0.4 &  19 &    0.3 &  19\\
DER-24           &  1024 & 20 &    0.1 &  27 & \textbf{   0.1} &  23 &    0.1 &  21 &    0.5 &  23 &    0.3 &  23\\
DER-24           &  1024 & 30 &    0.1 &  28 & \textbf{   0.1} &  22 &    0.1 &  24 &    0.8 &  26 &    0.3 &  26\\
DER-24           &  1024 & 40 &    0.2 &  44 & \textbf{   0.2} &  27 &    0.2 &  28 &    1.2 &  37 &    0.4 &  39\\
DER-24           &  1024 & 43 &    0.2 &  33 &    0.2 &  27 & \textbf{   0.2} &  21 &    1.1 &  33 &    0.5 &  33\\
DER-24           &  2048 & 10 &    0.2 &  31 & \textbf{   0.1} &  21 &    0.1 &  20 &    0.8 &  23 &    0.3 &  23\\
DER-24           &  2048 & 20 &    0.3 &  30 & \textbf{   0.1} &  19 &    0.2 &  18 &    1.0 &  22 &    0.3 &  22\\
DER-24           &  2048 & 30 & \textbf{   0.3} &  29 &    0.3 &  28 &    0.3 &  20 &    1.4 &  30 &    0.5 &  30\\
DER-24           &  2048 & 40 &    0.4 &  48 & \textbf{   0.4} &  36 &    0.4 &  27 &    2.5 &  47 &    0.6 &  47\\
DER-24           &  4096 & 10 &    0.4 &  35 & \textbf{   0.2} &  20 &    0.3 &  19 &    1.3 &  28 &    0.5 &  28\\
DER-24           &  4096 & 20 &    0.8 &  39 &    0.4 &  23 & \textbf{   0.4} &  22 &    2.1 &  29 &    0.5 &  29\\
DER-24           &  4096 & 30 &    1.3 &  65 &    1.0 &  42 & \textbf{   0.7} &  29 &    5.4 &  58 &    0.9 &  56\\
DER-24           &  4096 & 40 &    1.1 &  38 &    1.1 &  32 & \textbf{   0.9} &  26 &    5.0 &  38 &    0.9 &  38\\
DER-24           &  4096 & 41 &    1.0 &  40 &    1.0 &  32 & \textbf{   0.8} &  26 &    4.7 &  38 &    1.0 &  38\\
DER-24           &  8192 & 10 &    0.9 &  32 & \textbf{   0.5} &  18 &    0.6 &  21 &    2.6 &  25 &    0.7 &  25\\
DER-24           &  8192 & 20 &    2.0 &  39 &    1.1 &  26 & \textbf{   1.1} &  21 &    5.9 &  34 &    1.1 &  34\\
DER-24           &  8192 & 30 &    2.9 &  62 & \textbf{   1.8} &  36 &    2.1 &  40 &   12.5 &  55 &    1.9 &  55\\
DER-24           &  8192 & 40 &    4.3 &  79 &    2.7 &  46 & \textbf{   2.4} &  38 &   19.0 &  67 &    2.6 &  68\\
DER-24           & 16384 & 10 &    2.5 &  47 &    1.3 &  26 &    1.7 &  26 &    9.0 &  39 & \textbf{   1.2} &  36\\
DER-24           & 16384 & 20 &    4.1 &  42 &    2.1 &  29 &    2.5 &  22 &   13.8 &  37 & \textbf{   2.0} &  37\\
DER-24           & 16384 & 30 &    5.4 &  55 &    3.4 &  36 &    3.6 &  26 &   23.1 &  48 & \textbf{   3.0} &  48\\
DER-24           & 16384 & 40 &   12.4 & 110 &   10.2 &  88 & \textbf{   6.0} &  51 &   57.6 & 100 &    6.7 & 100\\
DER-24           & 16384 & 42 &   10.8 &  93 &    5.3 &  48 & \textbf{   5.3} &  42 &   49.8 &  86 &    5.3 &  83\\
DER-24           & 32768 & 10 &    4.6 &  39 &    3.3 &  34 &    3.5 &  23 &   17.9 &  36 & \textbf{   2.2} &  34\\
DER-24           & 32768 & 20 &   11.0 &  53 &    8.8 &  52 &    8.0 &  39 &   47.4 &  66 & \textbf{   5.5} &  65\\
DER-24           & 32768 & 30 &   14.5 &  68 &   12.3 &  56 & \textbf{   8.2} &  31 &   96.1 & 100 &   11.2 & 100\\
DER-24           & 32768 & 40 &   33.9 & 148 &   19.4 &  85 &   12.3 &  43 &  103.6 &  91 & \textbf{  11.4} &  86\\
DER-48           &  1024 & 10 &    0.1 &  24 & \textbf{   0.1} &  13 &    0.1 &  21 &    0.8 &  24 &    0.3 &  24\\
DER-48           &  1024 & 20 &    0.2 &  26 & \textbf{   0.2} &  20 &    0.2 &  22 &    1.1 &  26 &    0.3 &  26\\
DER-48           &  1024 & 30 &    0.3 &  31 & \textbf{   0.2} &  16 &    0.2 &  22 &    1.5 &  32 &    0.5 &  32\\
DER-48           &  1024 & 40 &    0.4 &  32 &    0.4 &  21 & \textbf{   0.3} &  22 &    1.6 &  30 &    0.4 &  30\\
DER-48           &  1024 & 49 &    0.5 &  37 &    0.4 &  19 & \textbf{   0.3} &  21 &    1.8 &  28 &    0.4 &  28\\
DER-48           &  2048 & 10 &    0.3 &  26 &    0.3 &  19 & \textbf{   0.3} &  20 &    1.3 &  24 &    0.4 &  25\\
DER-48           &  2048 & 20 &    0.7 &  37 &    0.5 &  21 &    0.5 &  22 &    2.2 &  31 & \textbf{   0.4} &  31\\
DER-48           &  2048 & 30 &    0.8 &  37 &    0.6 &  19 &    0.5 &  21 &    2.6 &  27 & \textbf{   0.5} &  27\\
DER-48           &  2048 & 40 &    1.2 &  38 &    0.9 &  24 &    0.7 &  21 &    4.1 &  33 & \textbf{   0.7} &  33\\
DER-48           &  2048 & 49 &    1.6 &  40 &    1.0 &  21 & \textbf{   0.8} &  25 &    5.7 &  37 &    0.9 &  37\\
DER-48           &  4096 & 10 &    0.8 &  34 &    0.6 &  19 &    0.6 &  21 &    3.2 &  28 & \textbf{   0.4} &  28\\
DER-48           &  4096 & 20 &    1.5 &  41 &    1.1 &  24 &    1.0 &  24 &    5.8 &  34 & \textbf{   0.8} &  34\\
DER-48           &  4096 & 30 &    1.7 &  38 &    1.4 &  23 &    1.4 &  23 &    8.2 &  35 & \textbf{   1.0} &  35\\
DER-48           &  4096 & 40 &    3.0 &  40 &    2.2 &  30 &    1.9 &  25 &    9.9 &  33 & \textbf{   1.4} &  33\\
DER-48           &  4096 & 49 &    4.1 &  44 &    2.0 &  25 &    2.0 &  27 &   14.1 &  39 & \textbf{   1.7} &  39\\
DER-48           &  8192 & 10 &    2.1 &  39 &    1.5 &  26 &    1.5 &  25 &    8.1 &  32 & \textbf{   1.1} &  32\\
DER-48           &  8192 & 20 &    3.9 &  44 &    1.9 &  18 &    2.0 &  23 &   12.7 &  31 & \textbf{   1.5} &  31\\
DER-48           &  8192 & 30 &    7.2 &  55 &    2.9 &  26 &    2.9 &  26 &   20.4 &  39 & \textbf{   2.2} &  39\\
DER-48           &  8192 & 40 &    7.3 &  45 &    4.0 &  24 &    3.8 &  22 &   25.4 &  38 & \textbf{   2.4} &  38\\
DER-48           &  8192 & 50 &    9.7 &  51 &    4.2 &  20 &    4.5 &  24 &   37.0 &  46 & \textbf{   3.4} &  47\\
DER-48           & 16384 & 10 &    5.0 &  49 &    2.9 &  25 &    3.8 &  35 &   22.4 &  41 & \textbf{   2.3} &  41\\
DER-48           & 16384 & 20 &    7.8 &  45 &    5.5 &  28 &    5.2 &  26 &   31.5 &  37 & \textbf{   2.8} &  37\\
DER-48           & 16384 & 30 &   14.6 &  59 &    7.3 &  29 &    6.3 &  25 &   53.6 &  50 & \textbf{   5.0} &  48\\
DER-48           & 16384 & 40 &   16.3 &  53 &    9.3 &  27 &    8.5 &  27 &   64.5 &  50 & \textbf{   5.2} &  45\\
DER-48           & 16384 & 48 &   22.3 &  64 &    9.9 &  29 &    9.3 &  28 &   89.7 &  60 & \textbf{   7.4} &  57\\
DER-48           & 32768 & 10 &   10.8 &  49 &    7.4 &  27 &    8.1 &  29 &   46.9 &  41 & \textbf{   4.1} &  41\\
DER-48           & 32768 & 20 &   16.8 &  47 &    8.6 &  24 &   11.2 &  32 &   69.5 &  42 & \textbf{   5.7} &  41\\
DER-48           & 32768 & 30 &   30.5 &  61 &   14.9 &  26 &   13.4 &  26 &  107.5 &  51 & \textbf{  10.0} &  51\\
DER-48           & 32768 & 40 &   36.2 &  57 &   21.1 &  31 &   16.9 &  28 &  133.8 &  51 & \textbf{  10.4} &  46\\
DER-48           & 32768 & 47 &   57.1 &  85 &   21.6 &  32 &   21.1 &  33 &  178.8 &  59 & \textbf{  14.2} &  54\\
DER-96           &  1024 & 10 &    0.5 &  27 &    0.3 &  18 & \textbf{   0.3} &  20 &    1.4 &  23 &    0.5 &  23\\
DER-96           &  1024 & 20 &    0.8 &  29 &    0.5 &  18 &    0.5 &  25 &    2.4 &  27 & \textbf{   0.4} &  27\\
DER-96           &  1024 & 30 &    1.2 &  32 &    0.6 &  17 &    0.6 &  23 &    3.5 &  30 & \textbf{   0.5} &  30\\
DER-96           &  1024 & 40 &    1.6 &  34 &    1.0 &  19 &    0.7 &  22 &    4.1 &  31 & \textbf{   0.6} &  32\\
DER-96           &  1024 & 50 &    2.2 &  34 &    1.2 &  19 &    0.8 &  22 &    5.8 &  30 & \textbf{   0.7} &  30\\
DER-96           &  1024 & 60 &    2.6 &  34 &    1.4 &  19 &    1.0 &  23 &    6.9 &  31 & \textbf{   0.8} &  31\\
DER-96           &  1024 & 64 &    3.3 &  38 &    1.2 &  19 &    0.9 &  22 &    6.6 &  31 & \textbf{   0.9} &  31\\
DER-96           &  2048 & 10 &    1.2 &  37 &    0.8 &  21 &    0.7 &  29 &    4.0 &  29 & \textbf{   0.5} &  29\\
DER-96           &  2048 & 20 &    2.2 &  33 &    1.1 &  19 &    0.9 &  23 &    5.9 &  26 & \textbf{   0.8} &  26\\
DER-96           &  2048 & 30 &    2.5 &  44 &    1.6 &  22 &    1.3 &  25 &    9.9 &  35 & \textbf{   1.1} &  35\\
DER-96           &  2048 & 40 &    4.8 &  38 &    2.3 &  26 &    1.5 &  23 &   12.4 &  33 & \textbf{   1.2} &  33\\
DER-96           &  2048 & 50 &    6.7 &  41 &    2.4 &  23 &    1.8 &  25 &   15.2 &  36 & \textbf{   1.6} &  36\\
DER-96           &  2048 & 56 &    7.9 &  45 &    2.4 &  20 &    1.7 &  21 &   18.2 &  38 & \textbf{   1.7} &  37\\
DER-96           &  4096 & 10 &    3.0 &  41 &    1.7 &  24 &    1.5 &  28 &    9.5 &  32 & \textbf{   1.0} &  32\\
DER-96           &  4096 & 20 &    4.4 &  51 &    2.9 &  27 &    1.9 &  26 &   18.2 &  39 & \textbf{   1.6} &  40\\
DER-96           &  4096 & 30 &    6.0 &  53 &    3.6 &  24 &    2.7 &  28 &   21.1 &  35 & \textbf{   2.0} &  36\\
DER-96           &  4096 & 40 &   13.6 &  53 &    4.4 &  24 &    3.3 &  27 &   31.2 &  39 & \textbf{   2.4} &  39\\
DER-96           &  4096 & 50 &   14.2 &  45 &    5.7 &  25 &    4.2 &  25 &   36.1 &  37 & \textbf{   2.7} &  37\\
DER-96           &  4096 & 53 &   16.3 &  51 &    5.5 &  24 &    5.2 &  28 &   42.6 &  40 & \textbf{   3.2} &  40\\
DER-96           &  8192 & 10 &    5.6 &  53 &    4.3 &  27 &    3.0 &  28 &   23.0 &  35 & \textbf{   2.6} &  35\\
DER-96           &  8192 & 20 &   11.1 &  62 &    7.2 &  33 &    4.9 &  32 &   43.4 &  40 & \textbf{   3.4} &  40\\
DER-96           &  8192 & 30 &   13.5 &  59 &    8.8 &  31 &    5.9 &  26 &   54.4 &  40 & \textbf{   3.8} &  40\\
DER-96           &  8192 & 40 &   32.7 &  63 &   12.6 &  35 &    7.4 &  25 &   77.6 &  45 & \textbf{   5.0} &  45\\
DER-96           &  8192 & 50 &   39.3 &  65 &   11.5 &  25 &   10.1 &  33 &   89.1 &  44 & \textbf{   6.6} &  45\\
DER-96           &  8192 & 60 &   51.7 &  75 &   15.5 &  29 &   11.1 &  31 &  137.6 &  60 & \textbf{   8.8} &  57\\
DER-96           & 16384 & 10 &   12.6 &  61 &   11.0 &  37 &    6.9 &  34 &   55.0 &  41 & \textbf{   4.4} &  41\\
DER-96           & 16384 & 20 &   21.2 &  62 &   14.5 &  31 &   10.9 &  31 &   92.3 &  42 & \textbf{   6.1} &  42\\
DER-96           & 16384 & 30 &   30.1 &  68 &   18.5 &  32 &   14.2 &  34 &  147.9 &  50 & \textbf{  10.1} &  52\\
DER-96           & 16384 & 40 &   70.0 &  69 &   21.8 &  28 &   16.1 &  30 &  196.5 &  54 & \textbf{  11.5} &  52\\
DER-96           & 16384 & 50 &   85.5 &  73 &   27.7 &  29 &   18.6 &  32 &  231.8 &  57 & \textbf{  14.5} &  54\\
DER-96           & 16384 & 57 &  107.8 &  86 &   31.9 &  31 &   24.4 &  39 &  260.0 &  55 & \textbf{  17.3} &  59\\
DER-96           & 32768 & 10 &   28.1 &  70 &   25.4 &  45 &   18.0 &  39 &  152.5 &  52 & \textbf{  11.8} &  49\\
DER-96           & 32768 & 20 &   39.9 &  57 &   33.8 &  36 &   18.9 &  28 &  180.4 &  37 & \textbf{  10.7} &  39\\
DER-96           & 32768 & 30 &   61.9 &  72 &   46.6 &  34 &   27.9 &  31 &  337.6 &  58 & \textbf{  21.3} &  58\\
DER-96           & 32768 & 40 &  174.6 &  88 &   70.8 &  42 &   40.4 &  39 &  483.2 &  69 & \textbf{  30.0} &  66\\
DER-96           & 32768 & 50 &  233.6 & 102 &   58.8 &  32 &   46.8 &  36 &  609.0 &  74 & \textbf{  43.1} &  72\\
DER-96           & 32768 & 54 &  291.9 & 119 &  102.9 &  54 & \textbf{  55.5} &  47 &  753.7 &  89 &   65.4 &  86\\
4node            &  1024 & 10 &    0.1 &  28 &    0.3 &  53 & \textbf{   0.1} &  28 &    0.9 &  31 &    0.5 &  30\\
4node            &  1024 & 20 &    0.2 &  27 &    0.2 &  22 & \textbf{   0.2} &  26 &    1.0 &  27 &    0.4 &  27\\
4node            &  1024 & 24 &   15.7 &  21 &    0.4 &  43 & \textbf{   0.2} &  25 &    1.3 &  30 &    0.5 &  32\\
4node            &  2048 & 10 &   19.6 &  24 &    0.7 &  51 & \textbf{   0.4} &  32 &    1.9 &  44 &    0.9 &  37\\
4node            &  2048 & 20 &    0.7 &  38 &    0.8 &  42 & \textbf{   0.5} &  37 &    1.9 &  33 &    0.9 &  32\\
4node            &  2048 & 24 &    0.7 &  38 &    0.6 &  27 & \textbf{   0.6} &  25 &    2.1 &  36 &    0.9 &  36\\
4node            &  4096 & 10 &    0.9 &  36 &    2.1 &  63 & \textbf{   0.7} &  28 &    3.6 &  40 &    1.3 &  40\\
4node            &  4096 & 20 &    0.9 &  23 &    1.0 &  27 & \textbf{   0.6} &  19 &    3.7 &  26 &    1.3 &  26\\
4node            &  4096 & 22 &    1.1 &  27 &    1.7 &  37 & \textbf{   0.7} &  17 &    4.5 &  28 &    1.2 &  28\\
4node            &  8192 & 10 & \textbf{   1.8} &  33 &    3.4 &  62 &    1.8 &  33 &   10.2 &  44 &    2.1 &  43\\
4node            &  8192 & 20 &    3.2 &  42 &    4.3 &  51 & \textbf{   2.3} &  36 &   14.2 &  43 &    3.2 &  44\\
4node            &  8192 & 23 &    2.7 &  30 &    2.3 &  29 & \textbf{   1.8} &  24 &   12.3 &  35 &    2.7 &  33\\
4node            & 16384 & 10 &    6.8 &  61 &   11.4 &  85 & \textbf{   5.7} &  53 &   34.4 &  62 &    5.7 &  67\\
4node            & 16384 & 20 &    7.0 &  42 &   20.8 & 108 &    6.4 &  44 &   31.6 &  45 & \textbf{   5.5} &  44\\
4node            & 16384 & 23 &    5.8 &  29 &   10.7 &  53 & \textbf{   4.0} &  22 &   26.4 &  33 &    4.6 &  33\\
4node            & 32768 & 10 &    9.8 &  42 &   11.8 &  42 &    9.1 &  40 &   56.4 &  52 & \textbf{   8.3} &  53\\
4node            & 32768 & 20 &   17.0 &  58 &   35.0 &  95 & \textbf{  13.5} &  45 &   81.9 &  60 &   15.7 &  65\\
4node            & 32768 & 21 &   17.0 &  57 &   18.7 &  55 &   14.6 &  41 &   74.8 &  56 & \textbf{  14.2} &  59\\
4node-base       &  1024 & 10 & \textbf{   0.1} &  24 &    0.2 &  18 &    0.3 &  21 &    0.8 &  24 &    0.3 &  24\\
4node-base       &  1024 & 20 &    0.3 &  25 &    0.3 &  23 & \textbf{   0.2} &  28 &    1.3 &  28 &    0.5 &  28\\
4node-base       &  1024 & 26 &   17.0 &  17 &    1.0 &  60 & \textbf{   0.3} &  27 &    1.4 &  28 &    0.6 &  27\\
4node-base       &  2048 & 10 &   15.0 &  15 &    0.8 &  36 & \textbf{   0.3} &  23 &    1.4 &  29 &    0.5 &  30\\
4node-base       &  2048 & 20 &    0.8 &  37 &    1.0 &  31 & \textbf{   0.7} &  35 &    3.1 &  36 &    0.8 &  36\\
4node-base       &  2048 & 27 &    1.0 &  35 &    2.6 &  72 & \textbf{   0.8} &  33 &    3.7 &  32 &    0.9 &  33\\
4node-base       &  4096 & 10 &    0.9 &  35 &    3.3 &  92 & \textbf{   0.7} &  24 &    4.4 &  38 &    0.9 &  42\\
4node-base       &  4096 & 20 &    1.8 &  38 &    4.4 &  76 & \textbf{   1.1} &  30 &    8.8 &  42 &    1.6 &  42\\
4node-base       &  4096 & 25 &    2.3 &  38 &    5.3 &  72 & \textbf{   1.5} &  34 &    9.1 &  34 &    1.8 &  34\\
4node-base       &  8192 & 10 &    1.8 &  33 &    2.0 &  21 &    1.5 &  26 &    7.6 &  29 & \textbf{   1.5} &  29\\
4node-base       &  8192 & 20 &    4.4 &  39 &   16.3 & 133 &    3.9 &  40 &   18.1 &  39 & \textbf{   2.7} &  39\\
4node-base       &  8192 & 22 &    3.8 &  29 &    3.7 &  27 & \textbf{   2.6} &  25 &   19.7 &  36 &    2.8 &  36\\
4node-base       & 16384 & 10 &    4.4 &  38 &   10.1 &  57 &    3.6 &  30 &   19.3 &  39 & \textbf{   3.4} &  39\\
4node-base       & 16384 & 20 &   10.6 &  49 &   17.1 &  51 &    6.9 &  36 &   46.2 &  46 & \textbf{   5.6} &  45\\
4node-base       & 16384 & 25 &   13.5 &  53 &   26.2 &  74 &    8.0 &  37 &   63.4 &  53 & \textbf{   7.0} &  47\\
4node-base       & 32768 & 10 &   10.9 &  45 &   76.1 & 214 &   10.3 &  40 &   44.3 &  36 & \textbf{   6.0} &  36\\
4node-base       & 32768 & 20 &   27.8 &  68 &   80.1 & 125 &   25.9 &  72 &  119.3 &  59 & \textbf{  15.6} &  63\\
4node-base       & 32768 & 23 &   20.3 &  37 &   29.0 &  43 &   14.9 &  30 &  107.7 &  48 & \textbf{  12.9} &  50\\
assets           & 37500 &  6 &    1.6 &  21 & \textbf{   0.6} &  12 &    1.1 &  20 &    2.0 &  13 &    1.0 &  13\\
env              &  1200 &  6 &    0.0 &  21 & \textbf{   0.0} &  12 &    0.1 &  16 &    0.3 &  16 &    0.3 &  16\\
env              &  1875 &  6 &    0.1 &  22 & \textbf{   0.0} &  12 &    0.1 &  13 &    0.4 &  16 &    0.4 &  16\\
env              &  3780 &  6 &    0.1 &  25 & \textbf{   0.1} &  12 &    0.1 &  14 &    0.7 &  17 &    0.5 &  17\\
env              &  5292 &  6 &    0.2 &  27 & \textbf{   0.1} &  13 &    0.1 &  13 &    0.7 &  17 &    0.7 &  17\\
env              &  8232 &  6 &    0.3 &  26 & \textbf{   0.2} &  13 &    0.3 &  14 &    1.1 &  18 &    1.2 &  18\\
env              & 32928 &  6 &    1.7 &  26 & \textbf{   0.9} &  13 &    1.3 &  17 &    5.1 &  21 &    4.4 &  21\\
env-diss         &  1200 & 10 &    0.0 &  17 & \textbf{   0.0} &  19 &    0.0 &  16 &    0.4 &  22 &    0.4 &  22\\
env-diss         &  1200 & 13 &    0.1 &  15 & \textbf{   0.0} &  15 &    0.1 &  17 &    0.4 &  23 &    0.4 &  23\\
env-diss         &  1875 & 10 &    0.1 &  27 & \textbf{   0.1} &  17 &    0.1 &  20 &    0.6 &  22 &    0.5 &  22\\
env-diss         &  1875 & 15 &    0.1 &  17 & \textbf{   0.1} &  18 &    0.1 &  18 &    0.6 &  22 &    0.5 &  22\\
env-diss         &  3780 & 10 &    0.2 &  23 & \textbf{   0.1} &  16 &    0.1 &  17 &    0.8 &  21 &    0.7 &  21\\
env-diss         &  3780 & 15 &    0.2 &  20 & \textbf{   0.1} &  18 &    0.2 &  18 &    1.0 &  22 &    1.0 &  22\\
env-diss         &  5292 & 10 &    0.4 &  31 & \textbf{   0.2} &  25 &    0.2 &  21 &    1.0 &  25 &    1.3 &  26\\
env-diss         &  5292 & 15 &    0.3 &  22 & \textbf{   0.3} &  23 &    0.3 &  22 &    1.3 &  25 &    1.5 &  25\\
env-diss         &  8232 & 10 &    0.6 &  26 &    0.4 &  22 & \textbf{   0.4} &  18 &    1.7 &  22 &    1.8 &  22\\
env-diss         &  8232 & 15 &    1.0 &  31 &    0.6 &  29 & \textbf{   0.5} &  23 &    3.2 &  35 &    2.6 &  35\\
env-diss         & 32928 & 10 &    4.6 &  37 &    2.8 &  36 & \textbf{   1.9} &  17 &    8.0 &  27 &    7.2 &  27\\
env-diss         & 32928 & 14 &    4.8 &  28 & \textbf{   2.1} &  22 &    2.5 &  19 &   10.0 &  27 &    7.7 &  27\\
phone            & 32768 &  5 &    0.5 &  15 & \textbf{   0.4} &   8 &    0.6 &   8 &    1.9 &  10 &    0.7 &  10\\
stormG2          &  1000 & 10 &    0.7 &  35 &    0.5 &  21 & \textbf{   0.3} &  21 &    2.0 &  32 &    1.5 &  31\\
stormG2          &  1000 & 20 &    1.4 &  33 &    0.8 &  18 & \textbf{   0.5} &  19 &    4.5 &  29 &    1.7 &  29\\
stormG2          &  1000 & 21 &    1.6 &  37 &    0.8 &  18 & \textbf{   0.5} &  22 &    4.0 &  29 &    1.7 &  28\\
\end{longtable}
}
}

\end{document}